\pdfoutput=1
\documentclass[11pt,a4paper,reqno]{amsart}

\usepackage[british]{babel}

\usepackage[DIV=9,oneside,BCOR=0mm]{typearea}
\usepackage{microtype}
\usepackage[T1]{fontenc}
\usepackage{setspace}
\usepackage{amssymb}
\usepackage{amsmath}
\usepackage{amsthm}
\usepackage{enumitem}
\usepackage{mathrsfs}
\usepackage[colorlinks,citecolor=blue,urlcolor=black,linkcolor=red,linktocpage]{hyperref}
\usepackage{mathrsfs}
\usepackage{mathtools}
\usepackage{xcolor}
\usepackage{xfrac}

\newtheorem{innercustomgeneric}{\customgenericname}
\providecommand{\customgenericname}{}
\newcommand{\newcustomtheorem}[2]{%
	\newenvironment{#1}[1]
	{%
		\renewcommand\customgenericname{#2}%
		\renewcommand\theinnercustomgeneric{##1}%
		\innercustomgeneric
	}
	{\endinnercustomgeneric}
}

\newcustomtheorem{customthm}{Theorem}

\newtheorem{thm}{Theorem}[section]

\newtheorem{cor}[thm]{Corollary}
\newtheorem{lem}[thm]{Lemma}
\newtheorem{prop}[thm]{Proposition}

\newtheorem*{claim}{Claim}
\newtheorem*{fact}{Fact}

\theoremstyle{definition}

\newtheorem{definition}[thm]{Definition}

\newtheorem{remark}[thm]{Remark}

\renewcommand{\epsilon}{\varepsilon}
\newcommand{\defeq}{\mathrel{\mathop:}=}

\DeclareMathOperator{\Orth}{O}
\DeclareMathOperator{\Sph}{\mathbb{S}}

\DeclareMathOperator{\trap}{T}
\DeclareMathOperator{\U}{U}
\DeclareMathOperator{\GL}{GL}
\DeclareMathOperator{\id}{id}
\DeclareMathOperator{\im}{im}
\DeclareMathOperator{\Aut}{Aut}
\DeclareMathOperator{\C}{\mathbb{C}}
\DeclareMathOperator{\R}{\mathbb{R}}
\DeclareMathOperator{\Q}{\mathbb{Q}}
\DeclareMathOperator{\Z}{\mathbb{Z}}
\DeclareMathOperator{\N}{\mathbb{N}}
\DeclareMathOperator{\T}{\mathbb{T}}
\DeclareMathOperator{\Sub}{Sub}

\DeclareMathOperator{\B}{B}
\DeclareMathOperator{\Pfin}{\mathcal{P}_{fin}}
\DeclareMathOperator{\Neigh}{\mathcal{U}}
\DeclareMathOperator{\Iso}{Iso}

\renewcommand{\Re}{\operatorname{Re}}

\makeatletter
\def\moverlay{\mathpalette\mov@rlay}
\def\mov@rlay#1#2{\leavevmode\vtop{%
		\baselineskip\z@skip \lineskiplimit-\maxdimen
		\ialign{\hfil$\m@th#1##$\hfil\cr#2\crcr}}}
\newcommand{\charfusion}[3][\mathord]{
	#1{\ifx#1\mathop\vphantom{#2}\fi
		\mathpalette\mov@rlay{#2\cr#3}
	}
	\ifx#1\mathop\expandafter\displaylimits\fi}
\makeatother

\newcommand{\bigveedot}{\charfusion[\mathop]{\bigvee}{\boldsymbol{\cdot}}}

\begin{document}

\setlist{noitemsep}

\author{Friedrich Martin Schneider}
\address{F.M.~Schneider, Institute of Discrete Mathematics and Algebra, TU Bergakademie Freiberg, 09596 Freiberg, Germany}
\email{martin.schneider@math.tu-freiberg.de}
\author{S{\l}awomir Solecki}
\address{S.~Solecki, Department of Mathematics, Cornell University, Ithaca, NY 14853,~USA}
\email{ssolecki@cornell.edu}
\thanks{The second-named author acknowledges funding of NSF grant DMS-2246873.}

\title[Exotic groups, submeasures, and escape dynamics]{Groups without unitary representations, submeasures, and the escape property}
\date{\today}

\begin{abstract} 
We give new examples of topological groups that do not have non-trivial continuous unitary representations, the so-called exotic groups. We prove that all groups of the form $L^0(\phi, G)$, where $\phi$ is a pathological submeasure and $G$ is a topological group, are exotic. This result extends, with a different proof, a theorem of Herer and Christensen on exoticness of $L^0(\phi,\R)$ for $\phi$ pathological. It follows that every topological group embeds into an exotic one. 
In our arguments, we introduce the escape property, a geometric condition on a topological group, inspired by the solution to Hilbert's fifth problem and satisfied by all locally compact groups, all non-archimedean groups, and all Banach--Lie groups. Our key result involving the escape property asserts triviality of all continuous homomorphisms from $L^0(\phi, G)$ to $L^0(\mu, H)$, where $\phi$ is pathological, $\mu$ is a measure, $G$ is a topological group, and $H$ is a topological group with the escape property.
\end{abstract}

\subjclass[2020]{22A10, 22A25, 28A60, 28B10}

\keywords{Topological group, unitary representation, exotic group, submeasure}

\maketitle


\tableofcontents

\section{Introduction}\label{section:introduction}

This is a paper about the relationship between representations of topological groups and submeasures. 

Among topological groups, we study a phenomenon specific to non-locally compact groups---nonexistence of continuous unitary representations. Groups with this property are called \emph{exotic}. Following their discovery in~\cite{HererChristensen}, further examples of exotic groups were found, including some Banach--Lie groups~\cite{banaszczyk,banaszczyk2}, the group of orientation-preserving homeomorphisms of the closed real unit interval endowed with the compact-open topology~\cite{megrel}, the isometry group of the Urysohn space of diameter one equipped with the topology of pointwise convergence~\cite[Corollary~1.4]{Pestov07}, as well as certain limits of finite special linear groups~\cite{CarderiThom}. An important feature of exotic groups is their relation to the well-studied class of \emph{extremely amenable}\footnote{A topological group $G$ is called \emph{amenable} (resp.,~\emph{extremely amenable}) if every continuous action of~$G$ on a non-void compact Hausdorf space admits an invariant regular Borel probability measure (resp.,~a fixed point). Exotic amenable topological groups are extremely amenable; see Remark~\ref{remark:whirly}.} groups, which have been intensively investigated in recent years and, aside from topological dynamics, have connections with logic, combinatorics, and probability theory. In fact, the first examples of extremely amenable groups were the exotic ones found in~\cite{HererChristensen}.

Among submeasures, we are interested in those that do not majorize non-trivial measures. Submeasures of this type are called \emph{pathological}. Pathological submeasures were discovered, explicitly and implicitly, several times in different mathematical contexts, as they are relevant to phenomena pertaining to combinatorics~\cite{ErdosHajnal}, measure theory~\cite{DaviesRogers}, topological dynamics~\cite{HererChristensen}, and set theory~\cite{mazur,talagrand}.

The connection between exotic groups and pathological submeasures  comes about through the $L^0$ construction. A submeasure $\phi$ can be thought of as a seminorm respecting the monoid structure, where $\phi$ is defined on a monoid of sets $\mathcal{A}$ with union as the monoid operation and with the empty set as the neutral element. If the monoid $\mathcal{A}$ happens to be closed under set difference, that is, if it is a Boolean ring of sets, then the seminorm gives rise to a pseudometric on $\mathcal{A}$, namely \begin{equation}\label{E:met}
	\mathcal{A} \times \mathcal{A} \ni (A,B) \, \longmapsto \, \phi \big((A\setminus B)\cup (B\setminus A)\big) \in \R . 
\end{equation} If $\mathcal{A}$ is actually a Boolean algebra and, additionally, a topological group $G$ is given, a pseudometric as in~\eqref{E:met} naturally lifts to a group topology on the group of $\mathcal{A}$-measurable $G$-valued simple functions, that is, functions that are measurable with respect to $\mathcal{A}$ and take only finitely many values, all of them in $G$. The $L^0$ group $L^0(\phi, G)$ is then obtained by taking a canonical completion of the group of $\mathcal{A}$-measurable $G$-valued simple functions endowed with the group topology induced by $\phi$ (Definition~\ref{definition:l0}). In the most familiar situation, when $\phi$ is a countably additive measure defined on a $\sigma$-algebra of sets and $G$ is a Polish group, the topological group $L^0(\phi, G)$ can be presented as the group of all $\phi$-equivalence classes of $\phi$-measurable $G$-valued functions with the topology of convergence in measure (Remark~\ref{remark:moore}). In Section~\ref{section:l0}, we describe the $L^0$ group construction in substantial generality and in full detail.  

The first example of an exotic group, found by Herer and Christensen~\cite{HererChristensen}, was the group $L^0(\phi,\R)$ with $\phi$ a pathological submeasure. The proof of exoticness of this group consisted of a mix of linear topological and descriptive set theoretic methods and relied crucially on $L^0(\phi,\R)$ being a topological $\R$-vector space, whose topology is Polish. In particular, the argument did not apply to $L^0(\phi, G)$ for $G$ much different from $\R$, not even for $G$ being the additive group of the rationals endowed with the discrete topology. While the exoticness theorem of Herer and Christensen has not been improved up to this point, its consequence saying that $L^0(\phi,\R)$ is extremely amenable, if $\phi$ is pathological, was generalized to larger classes of target groups, by different methods---combinatorial with roots in algebraic topology~\cite{FarahSolecki,sabok} and probabilistic~\cite{SchneiderSolecki}. Of those results, the most general one concerning a pathological submeasure $\phi$ asserts that $L^0(\phi, G)$ is extremely amenable if $G$ is an amenable topological group~\cite[Corollary~7.6]{SchneiderSolecki}. 

In the present paper, we prove exoticness of $L^0(\phi,G)$, with $\phi$ pathological, for an arbitrary topological group $G$. In fact, we establish more---the group $L^0(\phi,G)$ does not even have non-trivial continuous linear (thus, in particular, no non-trivial unitary) representations on a Hilbert space, a phenomenon named \emph{strong exoticness}.

\begin{customthm}{A}[Theorem~\ref{theorem:exotic}]\label{theorem:a} If $\phi$ is a pathological submeasure and $G$ is a topological group, then $L^{0}(\phi,G)$ is strongly exotic. \end{customthm}

Theorem~\ref{theorem:a} obviously extends the Herer--Christensen theorem. It also extends the extreme amenability results described above (Corollary~\ref{corollary:extreme.amenability}).  
It implies that every topological group naturally embeds into a strongly exotic one (Corollary~\ref{corollary:exotic.embedding}). Furthermore, it gives a characterization of pathological submeasures~$\phi$ in terms of properties of $L^{0}(\phi,G)$ (Corollary~\ref{corollary:exotic}).

While Theorem~\ref{theorem:a} is an extension of the Herer--Christensen theorem~\cite{HererChristensen}, our proof proceeds along different lines; neither does it follow the arguments in~\cite{FarahSolecki,sabok,SchneiderSolecki}.
An important step in it 
consists of establishing the following result on the nonexistence of non-trivial continuous homomorphisms from $L^{0}$ groups over pathological submeasures into $L^0$ groups over measures.  

\begin{customthm}{B}[Theorem~\ref{theorem:escape}]\label{theorem:b} If $G$ is a topological group, $H$ is a topological group with the escape property, $\phi$ is a pathological submeasure, and $\mu$ is a measure, then any continuous homomorphism from $L^{0}(\phi,G)$ to $L^{0}(\mu,H)$ is trivial. \end{customthm}

The \emph{escape property} (Definition~\ref{definition:escape}) is a hypothesis on the geometry of a topological group that has interesting connections with the solution of Hilbert's fifth problem, and the class of topological groups with the escape property contains all locally compact groups, all non-archimedean ones, and all Banach--Lie groups (Section~\ref{section:escape.dynamics}). It is clear that some restriction on $H$ is necessary to obtain a homomorphism rigidity result like Theorem~\ref{theorem:b}: as a matter of fact, for any submeasure $\phi$, any topological group $G$ and any non-zero measure $\mu$, the topological group $H \defeq L^{0}(\phi,G)$ admits a natural embedding into $L^{0}(\mu,H)$ (see Remark~\ref{remark:homomorphism} and Remark~\ref{remark:completion} for details). Thanks to Theorem~\ref{theorem:a}, the conclusion of Theorem~\ref{theorem:b} furthermore holds for any unitarily representable topological group $H$ (Corollary~\ref{corollary:exotic.3}). 

Beyond its relevance to the proof of Theorem~\ref{theorem:a}, we consider Theorem~\ref{theorem:b} to be interesting in its own right. Particularly appealing, for the simplicity of its statement, is the case where both $G$ and $H$ are equal to the two-element cyclic group. In this case, $L^0(\phi, G)$ and $L^0(\mu, H)$ become isomorphic to the completions of the topological groups $\mathcal{D}_{\phi}$ and $\mathcal{D}_{\mu}$ consisting of the domains of $\phi$ and $\mu$, each equipped with the symmetric difference as the group operation and endowed with the topology generated by the pseudometric induced by $\phi$ and $\mu$ as in~\eqref{E:met}. In this situation, we obtain a characterization of pathological submeasures as precisely those submeasures $\phi$ that do not admit non-trivial continuous group homomorphisms from $\mathcal{D}_{\phi}$ to $\mathcal{D}_{\mu}$ for any strictly positive measure $\mu$ (Corollary~\ref{corollary:boolean.groups}). In the same vein, our Theorem~\ref{theorem:a} entails that a submeasure $\phi$ is pathological if and only if $\mathcal{D}_{\phi}$ is (strongly) exotic (Corollary~\ref{corollary:exotic.2}).

This article is organized as follows. We define and discuss in Section~\ref{section:l0} some basic properties of $L^{0}$ groups, including certain natural extension mechanisms. In Section~\ref{section:escape.dynamics}, we introduce the escape property for topological groups, provide examples, and discuss connections with the geometry of $L^{0}$ groups. The subsequent Section~\ref{section:homomorphism.rigidity} is dedicated to homomorphism rigidity among topological $L^{0}$ groups and contains the proof of Theorem~\ref{theorem:b}. In Section~\ref{section:exoticness}, we prove our results concerning (strong) exoticness, in particular Theorem~\ref{theorem:a}. For the reader's convenience, we compile some relevant background information on submeasures in Appendix~\ref{appendix:submeasures} and on completions of topological groups in Appendix~\ref{appendix:completion}.

\subsection*{Notation}

\addtocontents{toc}{\protect\setcounter{tocdepth}{1}}

Throughout this paper, if $X$ is a set, then $\mathcal{P}(X)$ will denote the power set of~$X$, while $\Pfin(X)$ will denote the set of all finite subsets of $X$. Given a topological space $X$ and an element $x \in X$, we let $\Neigh_{x}(X)$ denote the neighborhood filter at $x \in X$. The neutral element of a group $G$ is denoted by $e = e_{G}$. If $G$ is a topological group, then we let $\Neigh(G) \defeq \Neigh_{e}(G)$.

\section{$L^{0}$ groups}\label{section:l0} 

The purpose of this section is to put together some definitions and basic facts concerning topological $L^{0}$ groups. These rather abstract objects, definable over any submeasure, have fairly concrete siblings in measure theory: topological groups of measurable maps. We start off by briefly recollecting the latter.

Let $(X,\mathcal{B})$ be a finite measure space and let $Y$ be a separable metrizable topological space. Consider the set $L^{0}(X,\mathcal{B},\mu;Y)$ of all equivalence classes of Borel measurable functions from $(X,\mathcal{B})$ to $Y$ up to equality $\mu$-almost everywhere. If $d$ is a metric generating the topology of $Y$, then \begin{align*}
	d_{\mu}^{0} \colon \, L^{0}(X,\mathcal{B},\mu;Y)^{2} \, &\longrightarrow \, \R, \\
	 (f,g) \, &\longmapsto \, \inf \{ \epsilon \in \R_{>0} \mid \mu(\{ x \in X \mid d(f(x),g(x)) > \epsilon \}) \leq \epsilon \}
\end{align*} constitutes a metric on $L^{0}(X,\mathcal{B},\mu;Y)$, and the topology on $L^{0}(X,\mathcal{B},\mu;Y)$ generated by $d_{\mu}^{0}$ coincides with the one generated by $\delta^{0}_{\mu}$ for any other metric $\delta$ generating the topology of $Y$ (see~\cite[Proposition~6]{moore} and \cite[p.~6, Corollary]{moore}). The resulting topology on $L^{0}(X,\mathcal{B},\mu;Y)$, which is uniquely determined by $\mu$ and the topology of $Y$, will be referred to as the \emph{topology of convergence in measure with respect to $\mu$}. Furthermore, if $d$ is a complete metric generating the topology of $Y$, then a standard application\footnote{For instance, the proof given in~\cite[Theorem~2.30, p.~61--61]{folland} for $Y=\R$ extends naturally.} of the Borel--Cantelli lemma shows that the metric $d_{\mu}^{0}$ is complete, too. Therefore, complete metrizability of $Y$ implies complete metrizability of $L^{0}(X,\mathcal{B},\mu;Y)$. What is more, if the $\sigma$-algebra $\mathcal{B}$ is countably generated and $Y$ is Polish, then $L^{0}(X,\mathcal{B},\mu;Y)$ is Polish by~\cite[Proposition~7]{moore}. Finally, if $G$ is a second-countable topological group, then $L^{0}(X,\mathcal{B},\mu;G)$, endowed with the pointwise multiplication (of representatives of equivalence classes) and the topology of convergence in measure with respect to $\mu$, constitutes a topological group.

We now proceed to defining an analogue of the above for submeasures on Boolean algebras instead of finite measure spaces. Of course, if $(X,\mathcal{B},\mu)$ is a finite measure space and $G$ is a Polish group, then the subgroup of $\mu$-equivalence classes of finitely valued measurable functions is dense in $L^{0}(X,\mathcal{B},\mu;G)$, which suggests thinking of the latter as a completion of the former and extending this definition to submeasures. This idea is captured in the following abstract approach to topological \emph{$L^{0}$ groups}, which is due to Fremlin~\cite[493Y]{Fremlin4}. We start off with the algebraic setup.

Let $G$ be a group and let $\mathcal{A}$ be a Boolean algebra. By a \emph{finite $G$-partition of unity in $\mathcal{A}$} we mean a map $a \in \mathcal{A}^{G}$ such that \begin{enumerate}
	\item[---] $\{ g \in G \mid a(g) \ne 0 \}$ is finite,
	\item[---] $\bigvee_{g \in G} a(g) = 1$, and
	\item[---] $a(g) \wedge a(h) = 0$ for any two distinct $g,h \in G$.
\end{enumerate} The set of all finite $G$-partitions of unity in $\mathcal{A}$ will be denoted by $S(\mathcal{A},G)$. Equivalently, the elements of $S(\mathcal{A},G)$ may be viewed as finite partitions of unity in $\mathcal{A}$ labeled by pairwise distinct elements of $G$. It is straightforward to verify that $S(\mathcal{A},G)$ constitutes a group with respect to the multiplication defined by \begin{align*}
	(a b)(g) \, &\defeq \, \bigvee \{ a(x)\wedge b(y) \mid x,y \in G, \, xy = g \} \\
	&= \, \bigvee\nolimits_{h \in G} a(h) \wedge b\!\left(h^{-1}g\right) \, = \, \bigvee\nolimits_{h \in G} a\!\left(gh^{-1}\right) \! \wedge b(h)
\end{align*} for $a,b \in S(\mathcal{A},G)$ and $g \in G$. The neutral element $e_{S(\mathcal{A},G)} \in S(\mathcal{A},G)$ is determined by $e_{S(\mathcal{A},G)}(e) = 1$ and $e_{S(\mathcal{A},G)}(g) = 0$ whenever $g \in G \setminus \{ e\}$.

By Stone's representation theorem~\cite{stone}, any Boolean algebra is isomorphic to a Boolean subalgebra of the powerset algebra of some set, and for such concrete algebras of sets, the groups defined above admit an equivalent description as groups of simple functions (see~\cite[493Y(a)(iii)]{Fremlin4} or~\cite[Remark~7.1]{SchneiderSolecki}).

\begin{remark}\label{remark:support} Let $\mathcal{A}$ be a Boolean algebra and let $G$ be a group. For each $a \in S(\mathcal{A},G)$, \begin{displaymath}
	\mathcal{P}(G) \, \longrightarrow \, \mathcal{A}, \quad T \, \longmapsto \, a[T] \defeq \bigvee\nolimits_{g \in T} a(g)
\end{displaymath} is a Boolean algebra homomorphism. If $S,T \subseteq G$, then routine calculations show \begin{itemize}
	\item[(i)] $e_{S(\mathcal{A},G)}[T] = \begin{cases} \, 1 & \text{if } e \in T, \\ \, 0 & \text{otherwise}, \end{cases}$
	\item[(ii)] $a^{-1}[T] = a\!\left[T^{-1}\right]$ for each $a \in S(\mathcal{A},G)$,
	\item[(iii)] $a[S] \wedge a[T] \leq ab[ST]$ for all $a,b \in S(\mathcal{A},G)$.
\end{itemize} \end{remark}

Submeasures give rise to compatible topologies on the groups introduced above. Let $G$ be a topological group and let $\phi$ be a submeasure on a Boolean algebra $\mathcal{A}$. Using Remark~\ref{remark:support} and~\cite[III, \S1.2, Proposition~1, p.~222--223]{bourbaki}, it is straightforward to verify that there is unique topology on $S(\mathcal{A},G)$, compatible with the group structure of $S(\mathcal{A},G)$, for which the family of sets 
\begin{displaymath}
	N_{\phi}(U,\epsilon)  \, \defeq\, \{ a \in S(\mathcal{A},G) \mid \phi (a[G\setminus U]) \leq \epsilon \} \qquad (U \in \Neigh(G), \, \epsilon \in \R_{>0}) 
\end{displaymath} constitutes a neighborhood basis at the neutral element. We refer to this topology as the \emph{$\phi$-topology} and we denote the resulting topological group by $S(\phi, G)$.

\begin{definition}\label{definition:l0} Let $G$ be a topological group and let $\phi$ be a submeasure. Then \begin{displaymath}
	L^{0}(\phi,G) \, \defeq \, \widehat{S(\phi,G)},
\end{displaymath} i.e., the topological group $L^{0}(\phi,G)$ is defined as the Ra\u{\i}kov completion, which consists of the set of all minimal bilaterally Cauchy filters on $S(\phi,G)$, equipped with the group structure and topology as defined in Appendix~\ref{appendix:completion}. Furthermore, we have the continuous homomorphism \begin{displaymath}
	\iota_{S(\phi,G)} \colon \, S(\phi,G) \, \longrightarrow \, L^{0}(\phi,G) , \quad a \, \longmapsto \, \Neigh_{a}(S(\phi,G)),
\end{displaymath} as defined in Remark~\ref{remark:completion}, which induces a topological group isomorphism \begin{displaymath}
	S(\phi,G)/\,\overline{\!\{e_{S(\phi,G)} \}\!} \, \longrightarrow \, \im \!\left(\iota_{S(\phi,G)}\right), \quad a \, \overline{\!\{e_{S(\phi,G)} \}\!}\, \, \longmapsto \, \iota_{S(\phi,G)}(a) .
\end{displaymath} By a slight abuse of notation, we will often identify $S(\phi, G)$ with the image of $\iota_{S(\phi,G)}$, thereby regarding $S(\phi, G)$ as a topological subgroup of $L^0(\phi, G)$.\footnote{This is in perfect alignment with the standard convention of identifying measurable functions with their equivalence classes relative to a given measure.} Since $S(\phi,G)$ is dense in $L^{0}(\phi,G)$ by Remark~\ref{remark:completion}, the collection of sets \begin{displaymath}
	\left. \left\{ \overline{N_{\phi}(U,\epsilon)} \, \right\vert U \in \Neigh(G), \, \epsilon \in \R_{>0} \right\} \!,
\end{displaymath} with the closures taken in $L^0(\phi, G)$, is a neighborhood basis at the identity in $L^{0}(\phi,G)$. \end{definition}

\begin{remark}\label{remark:moore} Let $(X,\mathcal{B},\mu)$ be a finite measure space and let $G$ be a Polish group. Using Remark~\ref{remark:completion}, it is straightforward to verify that the map \begin{displaymath}
	S(\mu,G) \, \longrightarrow \, L^{0}(X,\mathcal{B},\mu;G), \quad a \, \longmapsto \, \tilde{a}
\end{displaymath} with \begin{displaymath}
	\tilde{a}\vert_{a(g)} \, \equiv_{\mu} \, g \qquad (a \in S(\mu,G), \, g \in G)
\end{displaymath} is a topological group embedding. 
Also, $\{ \tilde{a} \mid a \in S(\mu,G) \}$ is dense in $L^{0}(X,\mathcal{B},\mu;G)$. Since $L^{0}(X,\mathcal{B},\mu;G)$ is completely metrizable and therefore Ra\u{\i}kov complete by Remark~\ref{remark:isometry.groups.complete}(ii), the map above extends to a topological isomorphism \begin{displaymath}
	L^{0}(\mu,G) \, \longrightarrow \, L^{0}(X,\mathcal{B},\mu;G) ,
\end{displaymath} thanks to Remark~\ref{remark:completion} and Proposition~\ref{proposition:extension.to.completion}(iii). We will use this isomorphism to make the identification 
\begin{displaymath}
	L^{0}(\mu,G) \, \cong \, L^{0}(X,\mathcal{B},\mu;G) 
\end{displaymath} 
in Lemma~\ref{lemma:abstract.key}, in Lemma~\ref{lemma:escape}, and in the proof of Theorem~\ref{theorem:escape}, each time for $G=\R$, as well as in Lemma~\ref{lemma:alternative} for a general Polish group. 
\end{remark}

\begin{remark}\label{remark:boolean.transformations} Let $G$ be a topological group, $\mu$ be a measure on a Boolean algebra $\mathcal{A}$. \begin{itemize}
	\item[(i)] Suppose that $\nu$ is a measure on a Boolean algebra $\mathcal{B}$ and $\theta \colon \mathcal{A} \to \mathcal{B}$ is a Boolean algebra homomorphism with $\mu = \nu \circ \theta$. Then the map \begin{displaymath}
		\qquad S(\mu,G) \, \longrightarrow \, S(\nu,G), \quad a \, \longmapsto \, \theta \circ a
	\end{displaymath} induces a topological group embedding $\Theta \colon L^{0}(\mu,G) \to L^{0}(\nu,G)$, thanks to Remark~\ref{remark:completion} and Proposition~\ref{proposition:extension.to.completion}(iii). If $\theta$ is onto, then $\Theta$ is an isomorphism.
	\item[(ii)] Applying~(i) to the homomorphism $\theta \colon \mathcal{A} \to \mathcal{A}/\mathcal{N}_{\mu}, \, A \mapsto A \mathbin{\triangle} \mathcal{N}_{\mu}$ and the measure $\mu' \colon \mathcal{A}/\mathcal{N}_{\mu} \to \R, \, A \mathbin{\triangle} \mathcal{N}_{\mu} \mapsto \mu(A)$, we see that $L^{0}(\mu,G) \cong L^{0}(\mu',G)$.
\end{itemize} \end{remark}

For the purposes of understanding continuous homomorphisms into Ra\u{\i}kov complete Hausdorff topological groups (such as the unitary group of a Hilbert space endowed with the strong operator topology; see Remark~\ref{remark:isometry.groups.complete}(i)), taking the Ra\u{\i}kov completion of the domain is an inessential operation, as follows from Proposition~\ref{proposition:extension.to.completion}. 
So, the topological information relevant to such homomorphisms defined on $L^{0}$ groups is contained in the topological groups of labeled finite partitions of unity. The remainder of this section is dedicated to some basic properties of those groups. 

\begin{remark}\label{remark:homomorphism} Let $\phi$ be a submeasure on a Boolean algebra $\mathcal{A}$ and let $G$ be a topological group. The map \begin{displaymath}
	\eta_{\phi,G} \colon \, G \, \longrightarrow \, S(\phi,G) 
\end{displaymath} defined by \begin{displaymath}
	\eta_{\phi,G}(g) \colon \, G \, \longrightarrow \, \mathcal{A}, \quad x \, \longmapsto \, \begin{cases}
						\, 1 & \text{if } x=g, \\
						\, 0 & \text{otherwise}
					\end{cases} \qquad (g \in G)
\end{displaymath} is a continuous homomorphism. If $\phi$ is non-zero, then $\eta_{\phi,G}$ is an embedding of topological groups. \end{remark}

\begin{remark}\label{remark:L0.metrizable} If $\phi$ is a submeasure and $G$ is a discrete group, then the topology of $S(\phi,G)$ is generated by the bi-invariant pseudometric \begin{displaymath}
	d_{\phi} \colon \, S(\phi,G) \times S(\phi,G) \, \longrightarrow \, \R, \quad (a,b) \, \longmapsto \, \phi\!\left( ab^{-1}[G\setminus \{e\}]\right) .
\end{displaymath} \end{remark}

Let us point out an elementary, but useful extension property for homomorphisms.

\begin{lem}\label{lemma:exponential} Let $\phi$ be a submeasure. Let $G$ and $H$ be groups and let $\pi \colon G \to S(\phi,H)$ be a homomorphism. Then 
$\pi_{\#} \colon (S(\phi,G),d_{\phi}) \to (S(\phi,H),d_{\phi})$ defined by \begin{displaymath}
	\pi_{\#}(a)(h) \, \defeq \, \bigvee\nolimits_{g \in G} a(g) \wedge \pi(g)(h) \qquad (a \in S(\phi,G), \, h \in H)
\end{displaymath} is a $1$-Lipschitz homomorphism such that $\pi = {\pi_{\#}} \circ {\eta_{\phi,G}}$. \end{lem}

\begin{proof} First we show $\pi_{\#}(a) \in S(\phi,H)$ for $a \in S(\phi,G)$. Since $E \defeq \{ g \in G \mid a(g) \ne 0 \}$ is a finite set and $(\{ h \in H \mid \pi(a)(h) \ne 0 \})_{g \in E}$ is a family of finite sets and, moreover, \begin{align*}
	\{ h \in H \mid \pi_{\#}(a)(h) \ne 0 \} \, &= \, \bigcup\nolimits_{g \in G} \{ h \in H \mid a(g) \wedge \pi(g)(h) \ne 0 \} \\
		& \subseteq \, \bigcup\nolimits_{g \in E} \{ h \in H \mid \pi(g)(h) \ne 0 \} ,
\end{align*} we conclude that $\{ h \in H \mid \pi_{\#}(a)(h) \ne 0 \}$ is finite. Furthermore, if $h,h' \in H$ are distinct, then \begin{align*}
	\pi_{\#}(a)(h) \wedge \pi_{\#}(a)(h') \, &= \, \!\left( \bigvee\nolimits_{g\in G} a(g) \wedge \pi(g)(h) \right) \wedge \left( \bigvee\nolimits_{g'\in G} a(g') \wedge \pi(g')(h') \right) \\
		& = \, \bigvee\nolimits_{g\in G} a(g) \wedge \pi(g)(h) \wedge \pi(g)(h') \, = \, 0 .
\end{align*} Finally, \begin{align*}
	\bigvee\nolimits_{h \in H} \pi_{\#}(a)(h) \, &= \, \bigvee\nolimits_{h \in H} \bigvee\nolimits_{g \in G} a(g) \wedge \pi(g)(h) \\
		& = \, \bigvee\nolimits_{g \in G} a(g) \wedge \! \left( \bigvee\nolimits_{h \in H} \pi(g)(h) \right) \! \, = \, \bigvee\nolimits_{g \in G} a(g) \, = \, 1 .
\end{align*} Thus, $\pi_{\#}(a) \in S(\phi,H)$. 

Now, if $a,b \in S(\phi,G)$, then \begin{align*}
	\pi_{\#}(ab)(h) \, &= \, \bigvee\nolimits_{g \in G} (ab)(g) \wedge \pi(g)(h) \, = \, \bigvee\nolimits_{x,y \in G} a(x) \wedge b(y) \wedge \pi(xy)(h) \\
		&= \, \bigvee\nolimits_{x,y \in G} a(x) \wedge b(y) \wedge (\pi(x)\pi(y))(h) \\
		& = \, \bigvee\nolimits_{x,y \in G} \bigvee\nolimits_{x',y' \in H, \, x'y'=h} a(x) \wedge b(y) \wedge \pi(x)(x') \wedge \pi(y)(y') \\
		& = \, \bigvee\nolimits_{x',y' \in H, \, x'y'=h} \left( \bigvee\nolimits_{x \in G} a(x) \wedge \pi(x)(x') \right) \wedge \left( \bigvee\nolimits_{y \in G} b(y) \wedge \pi(y)(y') \right) \\
		& = \, \bigvee\nolimits_{x',y' \in H, \, x'y'=h} \pi_{\#}(a)(x') \wedge \pi_{\#}(b)(y') \, = \, (\pi_{\#}(a)\pi_{\#}(b))(h)
\end{align*} for every $h \in H$, i.e., $\pi_{\#}(ab) = \pi_{\#}(a)\pi_{\#}(b)$. This means that $\pi_{\#}$ is a homomorphism. Additionally, if $g \in G$, then \begin{displaymath}
	\pi_{\#}(\eta_{\phi,G}(g))(h) \, = \, \bigvee\nolimits_{x \in G} \eta_{\phi,G}(g)(x) \wedge \pi(x)(h) \, = \, \pi(g)(h) 
\end{displaymath} for each $h \in H$, that is, $\pi(g) = \pi_{\#}(\eta_{\phi,G}(g))$. This shows that $\pi = {\pi_{\#}} \circ {\eta_{\phi,G}}$. Finally, if $a \in S(\phi,G)$, then $\pi(e)(h) = 0$ for each $h \in H\setminus \{ e\}$, hence \begin{align*}
	&\pi_{\#}(a)[H\setminus\{e\}] \, = \, \bigvee\nolimits_{h \in H\setminus \{ e \}} \pi_{\#}(a)(h) \, = \, \bigvee\nolimits_{h \in H\setminus \{ e \}} \bigvee\nolimits_{g \in G} a(g) \wedge \pi(g)(h) \\
	& \qquad \qquad = \, \bigvee\nolimits_{h \in H\setminus \{ e \}} \bigvee\nolimits_{g \in G\setminus \{ e \}} a(g) \wedge \pi(g)(h) \, \leq \, \bigvee\nolimits_{g \in G\setminus \{ e \}} a(g) \, = \, a[G\setminus \{e\}] .
\end{align*} Consequently, we get \begin{align*}
	d_{\phi}(\pi_{\#}(a),\pi_{\#}(b)) \, &= \, \phi\!\left( \pi_{\#}(a)\pi_{\#}(b)^{-1}[H\setminus \{e\}]\right) \\
		& = \, \phi\!\left( \pi_{\#}\!\left(ab^{-1}\right)[H\setminus \{e\}]\right) \, \leq \, \phi\!\left( ab^{-1}[G\setminus \{e\}]\right) \, = \, d_{\phi}(a,b) 
\end{align*} for all $a,b \in S(\phi,G)$, i.e., $\pi_{\#} \colon (S(\phi,G),d_{\phi}) \to (S(\phi,H),d_{\phi})$ is indeed $1$-Lipschitz. \end{proof}

We deduce the following reduction lemma, which will find several applications in Sections~\ref{section:homomorphism.rigidity} and~\ref{section:exoticness}.

\begin{lem}\label{lemma:countable.integral} Let $H$ be a group and let $\tau$ be an arbitrary topology on the set $H$. Then the following are equivalent. \begin{itemize}
	\item[(i)] For every topological group $G$ and every pathological submeasure $\phi$, every $\tau$-continuous homomorphism from $S(\phi,G)$ to $H$ is trivial.
	\item[(ii)] For every pathological submeasure $\phi$ defined on a countable Boolean algebra, every $\tau$-continuous homomorphism from $S(\phi,\Z)$ to $H$ is trivial.
\end{itemize} \end{lem}

\begin{proof} (i)$\Longrightarrow$(ii). This is trivial.
	
(ii)$\Longrightarrow$(i). Let $\phi$ be a pathological submeasure defined on a Boolean algebra $\mathcal{A}$. Suppose that $\theta \colon S(\phi,G) \to H$ is a $\tau$-continuous homomorphism. Let $a \in S(\phi,G)$. By Corollary~\ref{corollary:pathological.countable}, there exists a countable Boolean subalgebra $\mathcal{B} \leq \mathcal{A}$ such that $\phi' \defeq \phi\vert_{\mathcal{B}}$ is pathological and such that $\{ a(g) \mid g \in G \} \subseteq \mathcal{B}$. Thus, $a \in S(\phi',G)$. Furthermore, by Lemma~\ref{lemma:exponential} and Remark~\ref{remark:L0.metrizable}, the homomorphism $\pi \colon \Z \to S(\phi',G), \, i \to a^{i}$ gives rise to a continuous homomorphism $\pi_{\#} \colon S(\phi',\Z) \to S(\phi',G)$ such that \begin{displaymath}
	a \, \in \, \pi(\Z) \, = \, \pi_{\#}(\eta_{\phi',\Z}(\Z)) \, \subseteq \, \pi_{\#}(S(\phi',\Z)) .
\end{displaymath} Since $S(\phi',G)$ is a topological subgroup of $S(\phi,G)$, the restriction \begin{displaymath}
	{\theta\vert_{S(\phi',G)}} \colon \, S(\phi',G) \, \longrightarrow \, H
\end{displaymath} is $\tau$-continuous, therefore the homomorphism ${\theta\vert_{S(\phi',G)}} \circ {\pi_{\#}}$ is $\tau$-continuous. Consequently, ${\theta\vert_{S(\phi',G)}} \circ {\pi_{\#}}$ is trivial by~(ii). In particular, $\theta (a) = e$. \end{proof}

We note another extension lemma. 

\begin{lem}\label{lemma:lifting} Let $G$ and $H$ be topological groups, let $f \colon G \to H$ be bilaterally uniformly continuous, and let $\phi$ be a submeasure. Then \begin{displaymath}
	f_{\bullet} \colon \, S(\phi,G) \, \longrightarrow \, S(\phi,H), \quad a \, \longmapsto \, \!\left( h \mapsto a\!\left[ f^{-1}(h) \right] \right)
\end{displaymath} is bilaterally uniformly continuous. It has a unique bilaterally uniformly continuous extension, which we denote by the same letter, \begin{displaymath}
	f_{\bullet} \colon \, L^{0}(\phi,G) \, \longrightarrow \, L^{0}(\phi,H).
\end{displaymath} \end{lem}

\begin{proof} A straightforward application of Remark~\ref{remark:support} shows that $f_{\bullet}(a) \in S(\phi,H)$ for every $a \in S(\phi,G)$, i.e., the map $f_{\bullet}$ is indeed well defined. To prove uniform continuity, let $V \in \Neigh(H)$. Since $f$ is bilaterally uniformly continuous, there exists $U \in \Neigh(G)$ such that \begin{displaymath}
	\forall g_{0},g_{1} \in G \colon \quad g_{0}g_{1}^{-1}\!, \, g_{0}^{-1}g_{1} \in U \ \Longrightarrow \ f(g_{0})f(g_{1})^{-1}\!, \, f(g_{0})^{-1}f(g_{1}) \in V .
\end{displaymath} Hence, for all $a,b \in S(\phi,G)$, \begin{align*}
	\left(f_{\bullet}(a)f_{\bullet}(b)^{-1} \right)\![H\setminus V] \, 
		&= \, \bigvee \left\{ f_{\bullet}(a)(h_{0}) \wedge f_{\bullet}(b)(h_{1}) \left\vert \, h_{0},h_{1} \in H, \, h_{0}h_{1}^{-1} \notin V \right\} \right. \\
		&= \, \bigvee \left\{ a\!\left[f^{-1}(h_{0})\right] \! \wedge b\!\left[f^{-1}(h_{1})\right] \left\vert \, h_{0},h_{1} \in H, \, h_{0}h_{1}^{-1} \notin V \right\} \right. \\
		&= \, \bigvee \left\{ a(g_{0}) \wedge b(g_{1}) \left\vert \, g_{0},g_{1} \in G, \, f(g_{0})f(g_{1})^{-1} \notin V \right\} \right. \\
		&\leq \, \bigvee \left\{ a(g_{0}) \wedge b(g_{1}) \left\vert \, g_{0},g_{1} \in G, \, g_{0}g_{1}^{-1} \notin U \right\} \right. \\
		&\qquad \qquad \vee \bigvee \left\{ a(g_{0}) \wedge b(g_{1}) \left\vert \, g_{0},g_{1} \in G, \, g_{0}^{-1}g_{1} \notin U \right\} \right. \\
				&= \, ab^{-1}[G\setminus U] \vee a^{-1}b[G\setminus U]
\end{align*} and, analogously, \begin{displaymath}
	\left(f_{\bullet}(a)^{-1}f_{\bullet}(b) \right)\![H\setminus V] \, \leq \, ab^{-1}[G\setminus U] \vee a^{-1}b[G\setminus U].
\end{displaymath} The two inequalities together give \begin{align*}
	&\phi\!\left( \left(f_{\bullet}(a)f_{\bullet}(b)^{-1} \right)\![H\setminus V] \vee \left(f_{\bullet}(a)^{-1}f_{\bullet}(b) \right)\![H\setminus V] \right) \\
	&\qquad \qquad \leq \, \phi\!\left( ab^{-1}[G\setminus U] \vee a^{-1}b[G\setminus U] \right) \! \, \leq \, \phi\!\left( ab^{-1}[G\setminus U] \right) \! + \phi\!\left( a^{-1}b[G\setminus U] \right)\! .
\end{align*} It follows that, for every $\epsilon \in \R_{>0}$ and for all $a,b \in S(\phi,G)$, one has \begin{displaymath}
	\quad ab^{-1}\!, \, a^{-1}b \in N_{\phi}\!\left(U,\tfrac{\epsilon}{2}\right) \ \, \Longrightarrow \ \, f_{\bullet}(a)f_{\bullet}(b)^{-1}\!, \, f_{\bullet}(a)^{-1}f_{\bullet}(b) \in N_{\phi}(V,\epsilon) .
\end{displaymath} Therefore, $f_{\bullet}$ is bilaterally uniformly continuous. 
	
The existence of the extension $L^{0}(\phi,G) \to L^{0}(\phi,H)$ of $f_{\bullet}$ follows by Remark~\ref{remark:completion} and Proposition~\ref{proposition:extension.to.completion}(i). \end{proof}

Let us turn to a family of subgroups of groups of labeled finite partitions of unity. The set of all subgroups of a group $G$ will be denoted by $\Sub (G)$.
	
\begin{lem}\label{lemma:supported.subgroups} Let $\mathcal{A}$ be a Boolean algebra and let $G$ be a group. Then \begin{displaymath}
	\Gamma \defeq \Gamma_{\mathcal{A},G} \colon \, \mathcal{A} \, \longrightarrow \, \Sub(S(\mathcal{A},G)), \quad A \, \longmapsto \, \{ a \in S(\mathcal{A},G) \mid a[G\setminus \{e\}] \leq A \} 
\end{displaymath} is well defined. Moreover, \begin{itemize}
	\item[(i)] $\Gamma(0) = \{ e_{S(\mathcal{A},G)} \}$ and $\Gamma(1) = S(\mathcal{A},G)$,
	\item[(ii)] $\Gamma(A \vee B) = \Gamma(A) \Gamma(B)$ for all $A,B \in \mathcal{A}$.
\end{itemize} \end{lem}
	
\begin{proof} First of all, by Remark~\ref{remark:support}, the map $\Gamma$ is well defined, i.e., $\Gamma(A) \in \Sub(S(\mathcal{A},G))$ for every $A \in \mathcal{A}$.
		
(i) This is trivial. 
		
(ii) Let $A,B \in \mathcal{A}$. We observe that $\Gamma(A) \subseteq \Gamma(A \vee B)$ and $\Gamma(B) \subseteq \Gamma(A \vee B)$. Since $\Gamma(A \vee B)$ is a subgroup of $S(\mathcal{A},G)$, thus $\Gamma(A)\Gamma(B) \subseteq \Gamma(A \vee B)$. It remains to prove that $\Gamma(A \vee B) \subseteq \Gamma(A)\Gamma(B)$. To this end, let $c \in \Gamma(A \vee B)$. We define \begin{displaymath}
	a \colon \, G \, \longrightarrow \, \mathcal{A} , \quad g \, \longmapsto \, \begin{cases}
				\, c(e) \vee \neg A & \text{if } g=e, \\
				\, c(g) \wedge A & \text{otherwise}
			\end{cases}
\end{displaymath} as well as \begin{displaymath}
	b \colon \, G \, \longrightarrow \, \mathcal{A} , \quad g \, \longmapsto \, \begin{cases}
				\, c(e) \vee A & \text{if } g=e, \\
				\, c(g) \wedge \neg A & \text{otherwise}.
			\end{cases}
\end{displaymath} One readily sees that $a \in \Gamma(A)$ and $b \in \Gamma((A\vee B) \wedge \neg A ) \subseteq \Gamma(B)$. Furthermore, we observe that \begin{equation}\label{disjointness}
	\forall x,y \in G \setminus \{ e \} \colon \qquad a(x) \wedge b(y) \, = \, 0 .
\end{equation} We conclude that \begin{align*}
	(ab)(e) \, = \, \bigvee \{ a(x) \wedge b(y) \mid x,y \in G, \, xy=e \} \, \stackrel{\eqref{disjointness}}{=} \, a(e) \wedge b(e) 
			 =  \, c(e)
\end{align*} and, for each $g \in G \setminus \{ e \}$, \begin{align*}
	(ab)(g) \, &= \, \bigvee \{ a(x) \wedge b(y) \mid x,y \in G, \, xy=g \} \, \stackrel{\eqref{disjointness}}{=} \, (a(g) \wedge b(e)) \vee (a(e) \wedge b(g)) \\
		& = \, (c(g) \wedge A \wedge (c(e) \vee A)) \vee ((c(e) \vee \neg A) \wedge c(g) \wedge \neg A) \\
		& \stackrel{c(g) \wedge c(e) = 0}{=} \, (c(g) \wedge A) \vee (c(g) \wedge \neg A) \, = \, c(g) .
\end{align*} Therefore, $c = ab \in \Gamma(A)\Gamma(B)$ as desired. \end{proof}
	
Groups of labeled finite partitions of unity admit the following description as inductive limits of finite direct powers of the respective target groups.

\begin{remark}\label{remark:inductive.limit} Let $\mathcal{A}$ be a Boolean algebra and let $G$ be a group. For each $\mathcal{Q} \in \Pi(\mathcal{A})$, the injection $\sigma_{\mathcal{Q}} \colon G^{\mathcal{Q}} \to S(\mathcal{A},G)$ defined by \begin{displaymath}
	\sigma_{\mathcal{Q}}(g) \colon \, G \, \longrightarrow \, \mathcal{A}, \quad x \, \longmapsto \, \bigvee \{ Q \in \mathcal{Q} \mid g_{Q} = x \}  \qquad \left(g \in G^{\mathcal{Q}}\right)
\end{displaymath} is a homomorphism. It is easy to see that, if $\mathcal{Q},\mathcal{Q}' \in \Pi(\mathcal{A})$ and $\mathcal{Q} \preceq \mathcal{Q}'$, then \begin{displaymath}
	\sigma_{\mathcal{Q}}\!\left(G^{\mathcal{Q}}\right)\! \, \subseteq \, \sigma_{\mathcal{Q}'}\bigl(G^{\mathcal{Q}'}\bigr).
\end{displaymath} Since $(\Pi(\mathcal{A}),{\preceq})$ is a directed set by Remark~\ref{remark:directed.partitions}, thus \begin{displaymath}
	\mathcal{G} \, \defeq \, \left. \! \left\{ \sigma_{\mathcal{Q}}\!\left(G^{\mathcal{Q}}\right) \, \right\vert \mathcal{Q} \in \Pi(\mathcal{A}) \right\}
\end{displaymath} is a directed set of subgroups of $S(\mathcal{A},G)$. Moreover, $S(\mathcal{A},G) = \bigcup \mathcal{G}$. Well-known closure properties of the class of amenable groups (see, e.g.,~\cite[Theorem~12.4(d,\,e,\,f)]{Wagon}) hence entail that, if $G$ is amenable, then the group $S(\mathcal{A},G)$ is amenable, too.

In turn, for every $B \in \mathcal{A}$, considering the Boolean algebra $\mathcal{B} \defeq \{ A \in \mathcal{A} \mid A \leq B \}$ (equipped with the order inherited from $\mathcal{A}$), we conclude that amenability of $G$ implies amenability of $\Gamma_{\mathcal{A},G}(B) \cong S(\mathcal{B},G)$. \end{remark}

\section{Escape dynamics}\label{section:escape.dynamics}

In this section, we define the escape property for topological groups that is needed in Theorem~\ref{theorem:b}. We also 
work out some basic properties of this class of groups. 

We recall that a \emph{length function} on a group $G$ is a function $f \colon G \to \R$ such that \begin{itemize}
	\item[---] $f(e) = 0$,
	\item[---] $f\!\left(x^{-1}\right) = f(x) \geq 0$ for every $x \in G$, and
	\item[---] $f(xy) \leq f(x) + f(y)$ for all $x,y \in G$.
\end{itemize} 
The following well-known statement gives the starting point for the definition of our class of topological groups. 

\begin{fact} If $G$ is a topological group, then the set 
\begin{displaymath}
	\left. \left\{ f^{-1}([0,1)) \, \right\vert f \text{ continuous length function on } G \right\}
\end{displaymath} constitutes a neighborhood basis at the neutral element in $G$. \end{fact}

Indeed, if $U$ is an identity neighborhood in $G$, then Urysohn's lemma for uniform spaces (see, e.g.,~\cite[pp.~182--183]{james}) asserts the existence of a left-uniformly\footnote{A map $f \colon G \to \C$ is left-uniformly continuous if and only if, for every $\epsilon \in \R_{>0}$, there exists~$U \in \Neigh (G)$ such that $\sup \{ \vert f(gu)-f(g) \vert \mid g \in G, \, u \in U \} \leq \epsilon$.} continuous function $F \colon G \to [0,1]$ such that $F(e) = 1$ and $F^{-1}((0,1]) \subseteq U$, wherefore \begin{displaymath}
	f \colon \, G \, \longrightarrow \, [0,1], \quad x \, \longmapsto \, \sup\{ \vert F(gx) - F(g) \vert \mid g \in G \}
\end{displaymath} is a continuous length function on $G$ with $f^{-1}([0,1)) \subseteq U$. 

We define the desired class of topological groups by demanding that the conclusion of the statement above holds for an appropriately restricted subset of continuous length functions. The way this restriction is implemented is inspired by the notion of uniform freeness from small subgroups introduced by Enflo in~\cite{enflo}. We make more detailed comments on the relationship with~\cite{enflo} below. We borrow the following notation from~\cite{MorrisPestov}. For a group $G$, a subset $U \subseteq G$ and $n\in \N$, we set \begin{displaymath}
	\tfrac{1}{n}U \, \defeq \, \{ g \in G \mid g^{1},\ldots,g^{n} \in U \}. 
\end{displaymath} If $G$ is a topological group, then, for all $U \in \Neigh(G)$ and $n \in \N$, one has \begin{displaymath}
	\tfrac{1}{n}U \in \Neigh(G)\ \hbox{ and }\ \tfrac{1}{n}U\supseteq \tfrac{1}{n+1}U.
\end{displaymath} So, the procedure $(n, U)\mapsto \tfrac{1}{n}U$ uses the group structure to produce a shrinking sequence of identity neighborhoods. For later convenience, we also introduce the following piece of notation, inspired by Tao's account of the solution to Hilbert's fifth problem~\cite[Definition~5.4.1]{TaoBook}. For a group $G$ and a subset $U \subseteq G$, let\footnote{One may think of the elements of $\trap(U)$ as those being \emph{trapped} by $U$.}
\begin{displaymath}
	\trap(U) \, \defeq \, \bigcup \{ H \leq G \mid H \subseteq U \} \, = \, \{ g \in G \mid \langle g \rangle \subseteq U \}
\end{displaymath} that is, $\trap(U)$ is defined as the union of all subgroups of $G$ contained in $U$. Note that, if $e \in U = U^{-1}$, then $T(U) = \bigcap\nolimits_{n \in \N} \tfrac{1}{n}U$.

\begin{definition}\label{definition:escape} Let $G$ be a topological group. A function $f\colon G\to \R$ is called an \emph{escape function} if it is a length function on $G$ for which there exists an identity neighborhood $U \in \Neigh(G)$ 
such that \begin{displaymath}
	\forall \epsilon \in \R_{>0} \, \exists n \in \N\colon \quad \tfrac{1}{n}U \subseteq f^{-1}([0,\epsilon)) ;
\end{displaymath} in which case $U$ will be called an \emph{escape neighborhood for $f$}. We say that $G$ has the \emph{escape property} if the set \begin{displaymath}
	\left. \left\{ f^{-1}([0,1)) \, \right\vert f \text{ escape function on } G \right\}
\end{displaymath} constitutes a neighborhood basis at the neutral element in $G$. \end{definition} 

\begin{remark}\label{remark:escape.function} Any escape function on a topological group is necessarily continuous. \end{remark}

We compare having the escape property with uniform freeness from small subgroups defined by Enflo~\cite{enflo}.

\begin{remark}\label{remark:enflo} Let $G$ be a topological group. Then the following are equivalent. \begin{itemize}
	\item[(i)] $G$ is \emph{uniformly free from small subgroups}~\cite[\S2.1, p.~239]{enflo}, that is, \begin{displaymath}
		\qquad \exists U \in \Neigh(G) \, \forall V \in \Neigh(G) \, \exists n \in \N \colon \ \tfrac{1}{n}U \subseteq V .
	\end{displaymath}
	\item[(ii)] $G$ is first-countable and every continuous length function on $G$ is an escape function on $G$.
	\item[(iii)] There exists an escape function $f$ on $G$ such that $\left.\! \left\{ f^{-1}([0,\epsilon)) \, \right\vert \epsilon \in \R_{>0} \right\}$ is a neighborhood basis at the neutral element in $G$.
\end{itemize} Of course, (iii)$\Longrightarrow$(i) is trivial. The implication (ii)$\Longrightarrow$(iii) follows by applying the Birkhoff--Kakutani metrization theorem~\cite{birkhoff,kakutani} to $G/\,\overline{\!\{e \}\!}$ and then using the bijection between the set of left-invariant continuous pseudometrics on~$G$ and the set of continuous length functions on $G$ defined by \begin{displaymath}
	d \, \longmapsto \, [x \mapsto d(e,x)] .
\end{displaymath} The main part of (i)$\Longrightarrow$(ii) is the observation that uniform freeness from small subgroups implies first-countability, which is due to Enflo~\cite[Theorem~2.1.1]{enflo}. \end{remark}

Uniform freeness from small subgroups is strictly stronger than the escape property. While the former implies first-countability, the latter does not. But the difference between the two properties goes beyond matters of countability, as substantiated by Proposition~\ref{proposition:escape.groups}---even within the classes of Polish or locally compact groups, the escape property does not imply uniform freeness from small subgroups.

The following lemma contains the main closure property of the class of groups with the escape property. 

\begin{lem}\label{L:quotes} 
Let $G$ be a topological group. Suppose that there exists a subbasis $\mathcal S$ of the neighborhood filter at the identity in $G$ with the following property: for each~$S\in {\mathcal S}$, there exist $F\unlhd H\leq G$ such that \begin{itemize}
	\item[---] $F$ is  closed and $H$ is open in $G$,
	\item[---] $F\subseteq S$, 
	\item[---] the topological group $H/F$ has the escape property. 
\end{itemize} Then $G$ has the escape property. \end{lem}

\begin{proof} First, we make the following general observations. Let $G$ and $G'$ be topological groups and let $f'$ be a continuous length function on $G'$. \begin{itemize}
	\item[(a)] If $\pi \colon G \to G'$ is a continuous homomorphism, then $f=f' \circ \pi$ is a continuous length function on $G$. Furthermore,  if $U \in \Neigh(G')$ is an escape neighborhood for $f'$, then $\pi^{-1}(U) \in \Neigh(G)$ is an escape neighborhood for $f$.
	\item[(b)] Suppose that $G'$ is an open subgroup on $G$. Then \begin{displaymath}
					\qquad f \colon \, G \, \longrightarrow \, \R, \quad x \, \longmapsto \, \begin{cases}
																					\, f'(x) \wedge 1 & \text{if } x \in H, \\
																					\, 1 & \text{otherwise}
																				\end{cases}
				\end{displaymath} is a continuous length function on $G$. Furthermore, any escape neighborhood for $f'$ in $G'$ is an escape neighborhood for $f$ in $G$.
\end{itemize} 

Now, let us turn to the proof of the lemma. Fix $S\in {\mathcal S}$. Let $F \unlhd H$ be subgroups of $G$ that are, respectively, closed and open. Suppose that $F\subseteq S$ and $H/F$ has the escape property. Let $\pi\colon H\to H/F$ be the quotient homomorphism. Since $\pi$ is an open map and thus $\pi(S\cap H) \in \Neigh(H/F)$, there exists an escape function $f'$ on $H/F$ with $(f')^{-1}([0,1)) \subseteq \pi\big(S \cap H\big)$. Points~(a) and~(b) above imply that 
\begin{equation}\notag
	f \colon \, G \, \longrightarrow \, \R, \quad g \, \longmapsto \, 
	\begin{cases}
		\, f'(\pi(g)) \wedge 1 & \text{if } g \in G', \\
		\, 1 & \text{otherwise}
	\end{cases}
\end{equation} 
is an escape function on $G$, for which we obviously have \begin{displaymath}
	f^{-1}([0,1)) \, \subseteq \, \big(S \cap H\big)F \,\subseteq S^2.
\end{displaymath} 

Finally, let $U\in \Neigh(G)$. Then we find $V\in \Neigh(G)$ with $V^2\subseteq U$. Let $S_1, \ldots, S_n\in {\mathcal S}$ be such that $S_1\cap\ldots \cap S_n\subseteq V$. Evidently, $S_1^2\cap \ldots \cap S_n^2 \subseteq U$. By the argument in the previous paragraph there exist escape functions $f_1, \dots, f_n$ on~$G$ such that $f_i^{-1}([0,1))\subseteq S_i^2$ for each $i \in \{ 1,\ldots,n \}$. It is easy to see that $f \defeq f_1\vee \ldots \vee f_n$ is an escape function on $G$ and that $f^{-1}([0,1))\subseteq U$, which completes the proof. \end{proof} 

The following proposition, essentially a consequence of Lemma~\ref{L:quotes}, records some concrete operations preserving the escape property. In preparation, let $I$ be a set and consider the topological group ${\rm Sym}(I)$ consisting of all permutations of $I$, endowed with the topology of pointwise convergence induced by the discrete topology on $I$. The family of stabilizers of elements of $I$ forms a subbasis of the neighborhood filter at the identity in ${\rm Sym}(I)$. Now, if $\Sigma$ is a topological subgroup of ${\rm Sym}(I)$ and $G$ is any topological group, then \begin{displaymath}
	\Sigma \times G^{I} \! \, \longrightarrow \, G^{I}, \quad \left( \sigma, (g_{i})_{i \in I} \right) \, \longmapsto \, \bigl(g_{\sigma^{-1}(i)}\bigr)_{i \in I}
\end{displaymath} is a continuous action of $\Sigma$ on $G^{I}$ by automorphisms, wherefore the semidirect product $\Sigma\ltimes G^I$, whose underlying set is $\Sigma\times G^I$ and whose multiplication is defined by \begin{displaymath}
	(\sigma, (g_i)_{i\in I}) \cdot (\sigma', (g_i')_{i\in I}) \, \defeq \, (\sigma\sigma', (g_{\sigma'(i)}g'_i)_{i\in I}),
\end{displaymath} constitutes a topological group with respect to the product topology on $\Sigma\times G^I$. 

\begin{prop}\label{proposition.clos}
\begin{enumerate} 
	\item[(i)] Topological subgroups of topological groups with the escape property have the escape property. 
	\item[(ii)] Products of topological groups with the escape property have the escape property. 
	\item[(iii)] Let $I$ be a set. Let $\Sigma$ be a subgroup of ${\rm Sym}(I)$ and let $G$ be a topological group with the escape property. Then $\Sigma\ltimes G^I$ has the escape property. 
\end{enumerate} 
\end{prop}

\begin{proof} (i) This follows from the fact that the restriction of an escape function to a topological subgroup is an escape function on the latter. 

(ii) Consider a family $(G_{i})_{i \in I}$ of topological groups with the escape property. Applying Lemma~\ref{L:quotes} with the filter subbasis $\mathcal S$ consisting of all sets of the form \begin{displaymath}
	\left. \left\{ g\in \prod\nolimits_{i\in I} G_i \, \right\vert g_{j} \in U \right\}\!,
\end{displaymath} where $j \in I$ and $U\in \Neigh(G_{j})$, we see that $\prod_{i\in I} G_i$ has the escape property. Indeed, given a set in $\mathcal S$ as above, the conditions of Lemma~\ref{L:quotes} are easily verified for \begin{displaymath}
	F \, = \, \left. \! \left\{ g\in \prod\nolimits_{i\in I} G_i \, \right\vert g_{j} = e_{G_{j}} \right\}, \qquad  H \, = \, \prod\nolimits_{i\in I} G_i. 
\end{displaymath} 

(iii) This proof is similar to the proof of (ii). We apply Lemma~\ref{L:quotes}. Consider the subbasis $\mathcal S$ 
of the identity neighborhood filter in $\Sigma\ltimes G^I$ given by the sets of the form 
\begin{displaymath}
	\left. \left\{ (\sigma,g)\in \Sigma\times G^I \, \right\vert (\sigma(j)=j) \wedge (g_{j} \in U) \right\}\!,
\end{displaymath} where $j \in I$ and $U\in \Neigh(G)$. 
For a set in $\mathcal S$ as above, we take \begin{displaymath}
	\left. F=  \left\{ (\sigma,g)\in \Sigma\times G^I \,\right\vert (\sigma(j)=j) \wedge (g_{j} =e_G) \right\}\!, \; \;  \left. H= \left\{ (\sigma,g)\in \Sigma\times G^I \,\right\vert \sigma(j)=j \right\}\!,
\end{displaymath} and easily check the assumptions of Lemma~\ref{L:quotes} for this choice. 
\end{proof}

The next result describes the extent of the class of groups with the escape property. Recall that a topological group is \emph{non-archimedean} if it has a neighborhood basis at the identity consisting of open subgroups. 
If $X$ is a locally compact separable metric space, then the topological group ${\rm Iso}(X)$ of all (surjective) isometries of $X$, equipped with the topology of pointwise convergence, is Polish, that is, separable and completely metrizable; see \cite[Proposition~5.2]{GaoKechris}.
A topological group is called a \emph{group of isometries of a locally compact separable metric space} if it admits an embedding into ${\rm Iso}(X)$ for some locally compact separable metric space $X$. 

\begin{prop}\label{proposition:escape.groups} 
The following topological groups have the escape property: 
\begin{itemize}
	\item[(i)] Banach--Lie groups; 
	\item[(ii)] locally compact groups;
	\item[(iii)] non-archimedean topological groups;
	\item[(iv)] groups of isometries of locally compact separable metric spaces.
\end{itemize} 
\end{prop}

\begin{proof}
(i) Every Banach--Lie group is uniformly free from small subgroups, by a result of Morris and Pestov~\cite[Theorem~2.7]{MorrisPestov}, which had been also noted without proof by Enflo~\cite[p.~241, remark after Theorem~2.1.1]{enflo}. So, point (i) follows from Remark~\ref{remark:enflo}. 

(ii) Let $G$ be a locally compact group. Consider any $U \in \Neigh(G)$. Thanks to the Gleason--Yamabe theorem~\cite{gleason,yamabe}\footnote{See~\cite[Theorem~1.1.13]{TaoBook} for the statement used here.}, there exists a compact subgroup $K \leq G$ with $K \subseteq U$, the normalizer $N(K)$ open, and such that $N(K)/K$ is a Lie group. Thus, point (ii) follows from point (i) and Lemma~\ref{L:quotes} (with $F=K$ and $H=N(K)$). 

(iii) Let $G$ be a non-archimedean topological group. If $U \in \Neigh(G)$, then there exists an open, so also closed, subgroup $V \leq G$ with $V \subseteq U$. Thus, point (iii) follows from Lemma~\ref{L:quotes} (with $F=H=V$). 

(iv) Let $G$ be a group of isometries of a locally compact separable metric space. So, we can assume that $G$ is a topological subgroup of ${\rm Iso}(X)$ for a locally compact separable metric space. Since ${\rm Iso}(X)$ is Polish by~\cite[Proposition~5.2]{GaoKechris}, the closure $\overline{G}$ of $G$ in ${\rm Iso}(X)$ is also Polish. 
The characterization of Polish groups of isometries of locally compact separable metric spaces proved in \cite[Theorem~1.2]{KwiatkowskaSolecki} states that such groups are precisely those Polish groups, in which each identity neighborhood contains a closed subgroup $F$ with its normalizer $N(F)$ open in $G$ and $N(F)/F$ locally compact. Thus, $\overline{G}$ has the escape property by Lemma~\ref{L:quotes} and (ii). Point (iv) follows now from Proposition~\ref{proposition.clos}(i) as $G$ is a subgroup of $\overline{G}$. 
\end{proof}

Observe that the class of Banach--Lie groups includes, aside from all Lie groups, the unit group of any unital Banach algebra endowed with the norm topology, as well as the unitary group of any unital $C^{\ast}$-algebra equipped with the norm topology. It was proved by Rosendal~\cite[Theorem~3]{rosendal} that every topological group admitting a minimal metric is uniformly free from small subgroups, thus has the escape property.

The following result by Vladimir Pestov, which we include with his kind permission, provides a natural example of a topological group without the escape property.

\begin{prop}[Pestov]\label{proposition:pestov} Let $H$ be an infinite-dimensional Hilbert space\footnote{In this paper, all Hilbert spaces are complex. The statement and proof of Proposition~\ref{proposition:pestov}, however, remain valid in the real case.}. Then the unitary group $\U(H)$, endowed with the strong operator topology, does not have the escape property. \end{prop}

\begin{proof} Let $\Sph \defeq \{ x \in H \mid \Vert x \Vert_{H} = 1 \}$ and, for any closed linear subspace $V \leq H$, let $p_{V} \colon H \to H$ denote the orthogonal projection onto $V$.
	
Let $x \in \Sph$ and $U \defeq \{ g \in \U(H) \mid \Vert x-gx \Vert_{H} \leq \sqrt{2} \}$. We will prove that there is no escape function $f$ on $\U(H)$ satisfying $f^{-1}([0,1)) \subseteq U$. So, for contradiction, suppose that $f$ is an escape function on $\U(H)$ such that $f^{-1}([0,1)) \subseteq U$. Since $f$ is an escape function on $\U(H)$, there exist $\epsilon \in \R_{>0}$ and a finite subset $F \subseteq \Sph$ of pairwise orthogonal elements with $x \in F$ such that \begin{displaymath}
	V \, \defeq \, \{ g \in \U(H) \mid \forall z \in F \colon \, \Vert z - gz \Vert_{H} \leq \epsilon \} 
\end{displaymath} is an escape neighborhood for $f$. Let $\U(H)_{Z} \defeq \{ g \in \U(H) \mid \forall z \in Z \colon \, gz = z\}$ whenever $Z \subseteq H$. We observe that \begin{equation}\label{stabilizers}
	\forall (\tilde{z})_{z \in F} \in H^{F} \colon \quad \sup\nolimits_{z \in F} \Vert z-\tilde{z} \Vert_{H} \leq \tfrac{\epsilon}{2} \ \Longrightarrow \ \U(H)_{\{ \tilde{z} \mid z \in F\}} \subseteq f^{-1}(0) .
\end{equation} Indeed, if $(\tilde{z})_{z \in F} \in H^{F}$ and $\Vert z-\tilde{z} \Vert_{H} \leq \tfrac{\epsilon}{2}$ for each $z \in F$, then $\U(H)_{\{ \tilde{z} \mid z \in F\}} \subseteq V$ and therefore $\U(H)_{\{ \tilde{z} \mid z \in F\}} \subseteq \trap (V) \subseteq f^{-1}(0)$. Now, let us consider the closed linear subspace $L \defeq \sum_{z \in F\setminus \{ x\}} \C\! z \leq H$. Since $H$ is infinite-dimensional, we find $y \in \Sph \cap {L^{\perp}}$ such that $0 < \Vert x-y \Vert_{H} \leq \tfrac{\epsilon}{2}$ and $\langle x,y \rangle_{H} \in \R$. Consider the involutions \begin{align*}
	g_{x} \, &\defeq \, p_{\C \!x} - (p_{{L^{\perp}}}-p_{\C \!x}) + p_{L} \, = \, 2p_{\C \!x} - p_{{L^{\perp}}} + p_{L} \, \in \, \U(H), \\
	g_{y} \, &\defeq \, p_{\C \!y} - (p_{{L^{\perp}}}-p_{\C \!y}) + p_{L} \, = \, 2p_{\C \!y} - p_{{L^{\perp}}} + p_{L} \, \in \, \U(H),
\end{align*} and note that \begin{displaymath}
	g_{x} \, \in \, \U(H)_{F} \, \stackrel{\eqref{stabilizers}}{\subseteq} \, f^{-1}(0), \qquad g_{y} \, \in \, \U(H)_{\{ y\} \cup (F \setminus \{x\})} \, \stackrel{\eqref{stabilizers}}{\subseteq} \, f^{-1}(0) .
\end{displaymath} Since $f^{-1}(0)$ is a subgroup of $\U(H)$, it follows that $G \defeq \langle g_{x},g_{y} \rangle \subseteq f^{-1}(0)$. Furthermore, the fact that $\langle x,y \rangle_{H} \in \R$ entails that \begin{itemize}
	\item[---] the inner product of $H$ takes only real values on $E \defeq \R \!x + \R\! y$, thus turning $E$ into a two-dimensional Euclidean space;
	\item[---] $g_{x}(E) = E$ and $g_{y}(E) = E$.
\end{itemize} In turn, we obtain a well-defined homomorphism $\pi \colon G \to \Orth(E), \, g \mapsto g\vert_{E}$, where $\Orth(E)$ denotes the orthogonal group of $E$. Observe that $\pi(g_{x})$ (resp., $\pi(g_{y})$) is the reflection about the line through the origin and $x$ (resp., $y$). Hence, by a well-known elementary fact, $\pi(g_{x})\pi(g_{y})$ is a non-trivial rotation of $E$ about the origin. Consequently, there exists $n \in \N$ such that $(\pi(g_{x})\pi(g_{y}))^{n}$ is a rotation of $E$ about the origin through an angle from $\left(\tfrac{\pi}{2},\tfrac{3\pi}{2}\right)$. Concerning the unitary $g \defeq (g_{x}g_{y})^{n} \in G \subseteq f^{-1}(0) \subseteq f^{-1}([0,1))$, we conclude that \begin{displaymath}
	\Vert x - gx \Vert_{H} \, = \, \Vert x - \pi(g)x \Vert_{E} \, = \, \Vert x - (\pi(g_{x})\pi(g_{y}))^{n}x \Vert_{E} \, > \, \sqrt{2} ,
\end{displaymath} which constitutes the intended contradiction. \end{proof}

As already hinted at in the introductory Section~\ref{section:introduction}, more precisely in the paragraph just after Theorem~\ref{theorem:b}, topological $L^{0}$ groups over non-zero diffuse submeasures and with non-trivial target groups do not have the escape property---in fact, they do not admit any non-zero escape functions, as established in Proposition~\ref{proposition:power.bounded}.

\begin{lem}\label{lemma:power.bounded} Let $\phi$ be a diffuse submeasure and let $G$ be a topological group. For every $U \in \Neigh (S(\phi,G))$, there exists $n \in \N$ such that $\trap(U)^{n} = S(\phi,G)$.
\end{lem} 

\begin{proof} Let $\phi$ be defined on a Boolean algebra $\mathcal{A}$. Consider any $U \in \Neigh (S(\phi,G))$. Then there exist $V \in \Neigh (G)$ and $\epsilon \in \R_{>0}$ such that $N_{\phi}(V,\epsilon) \subseteq U$. Since $\phi$ is diffuse, there exists $\mathcal{Q} \in \Pi(\mathcal{A})$ such that $\sup \{ \phi(Q) \mid Q \in \mathcal{Q} \} \leq \epsilon$. Let $n \defeq \vert \mathcal{Q} \vert$. If $a \in S(\mathcal{A},G)$, then for each $Q \in \mathcal{Q}$ the mapping \begin{displaymath}
		a_{Q} \colon \, G \, \longrightarrow \, \mathcal{A}, \quad g \, \longmapsto \, \begin{cases}
			\, a(e) \vee \neg Q & \text{if } g=e, \\
			\, a(g) \wedge Q & \text{otherwise}
		\end{cases}
\end{displaymath} is a member of $\Gamma(Q)$, so that $\langle a_{Q} \rangle \subseteq \Gamma(Q) \subseteq N_{\phi}(V,\epsilon) \subseteq U$ and thus $a_{Q} \in \trap(U)$, and one easily checks that, with respect to any order on the factors in the product, \begin{displaymath}
	a \, = \, \prod\nolimits_{Q \in \mathcal{Q}} a_{Q} \, \in \, \trap(U)^{n}. \qedhere
\end{displaymath} \end{proof}

\begin{prop}\label{proposition:power.bounded} Let $\phi$ be a diffuse submeasure and let $G$ be a topological group. 
There is no non-zero escape function on $S(\phi,G)$ or $L^{0}(\phi,G)$.
\end{prop}

\begin{proof} Let $f$ be an escape function on $S(\phi,G)$ and let $U \in \Neigh(S(\phi,G))$ be an escape neighborhood for $f$. Then $f\vert_{\trap(U)} = 0$. Furthermore, by Lemma~\ref{lemma:power.bounded}, there exists~$n \in \N$ such that $\trap(U)^{n} = S(\phi,G)$. Since $f$ is a length function, it follows that $f=0$. 

Now let $f$ be an escape function on $L^{0}(\phi,G)$. Then $f\vert_{S(\phi, G)}$ is an escape function on~$S(\phi,G)$, thus equals $0$ by the above argument. Since $f$ is continuous by Remark~\ref{remark:escape.function} and $S(\phi, G)$ is dense in $L^{0}(\phi,G)$ according to Remark~\ref{remark:completion}, we get $f=0$. \end{proof}

Our proof of Theorem~\ref{theorem:escape} employs Lemma~\ref{lemma:escape}, which in turn builds on applying the lifting mechanism to escape functions from Lemma~\ref{lemma:lifting.2} below. In this lemma, we consider the group $S(\mu, \R)$ to be equipped with the natural binary relation $\leq$ defined as follows: for $a,b \in S(\mu,\R)$, \begin{displaymath}
	a\leq b \quad :\Longleftrightarrow \quad \forall r\in \R\colon \ a[[r,\infty)] \leq b[[r,\infty)], 
\end{displaymath} where $\leq$ on the right-hand side is, of course, the partial order on the Boolean algebra underlying $\mu$. It is easy to check that this relation is a partial order on $S(\mu, \R)$. 

\begin{lem}\label{lemma:lifting.2} Let $G$ be a topological, let $f$ be a continuous length function on $G$, and let $\mu$ be a measure. Let $a,b \in S(\mu,G)$ and $\epsilon\in \R_{>0}$. Then the following hold. \begin{itemize}
    \item[(i)] If $\mu\big( f_{\bullet}(a)[[\epsilon,\infty)] \big)\! \leq \epsilon$, then $a \in N_{\mu}\!\left(f^{-1}([0,\epsilon]),\epsilon\right)$.
	\item[(ii)] $f_{\bullet}\!\left( a^{-1}\right) = f_{\bullet}(a)$.
	\item[(iii)] $f_{\bullet}(ab) \leq f_{\bullet}(a) + f_{\bullet}(b)$.
	\item[(iv)] If $g \colon G \to [0,\infty)$ is bilaterally uniformly continuous and $g^{-1}([0,1))$ is an escape neighborhood for $f$, then there is $n\in {\mathbb N}$ such that, for each $c\in S(\mu, G)$, \begin{displaymath}
					\qquad f_{\bullet}(c)[[\epsilon,\infty)] \, \leq \, \bigvee\nolimits_{1\leq i \leq n} g_{\bullet}\!\left(c^i\right)\![[1,\infty)].
				\end{displaymath}
\end{itemize} \end{lem}

\begin{proof} (i) We get $a \in N_{\mu}\!\left(f^{-1}([0,\epsilon]),\epsilon\right)$ from \begin{displaymath}
	\mu\!\left( a\!\left[ G\setminus f^{-1}([0,\epsilon]) \right]\right)\! \, \leq \, \mu\!\left( a\!\left[ f^{-1}([\epsilon,\infty)) \right] \right)\! \, = \, \mu\big( f_{\bullet}(a)([\epsilon,\infty)) \big) \, \leq \, \epsilon.
\end{displaymath} 

(ii) For every $r \in \R$, \begin{displaymath}
	f_{\bullet}\!\left(a^{-1}\right)\!(r) \, = \, a^{-1}\!\left[ f^{-1}(r)\right]\! \, = \, a\!\left[ f^{-1}(r)^{-1} \right]\! \, = \, a\!\left[ f^{-1}(r) \right]\! \, = \, f_{\bullet}(a)(r) .
\end{displaymath} That is, $f_{\bullet}\!\left(a^{-1}\right) = f_{\bullet}(a)$. 
	
(iii) For every $r \in \R$, we have \begin{align*}
	f_{\bullet}(ab)[[r,\infty)]\, &= \, ab\!\left[ f^{-1}([r,\infty)) \right] \\
	& = \, \bigvee \{ a(x) \wedge b(y) \mid x,y \in G, \, f(xy) \geq r \} \\
	&\leq \, \bigvee \{ a(x) \wedge b(y) \mid x,y \in G, \, f(x)+f(y) \geq r \} \\
	&= \, \bigvee \{ f_{\bullet}(a)(s) \wedge f_{\bullet}(b)(t) \mid s,t \in \R, \, s+t \geq r \} \\
	&= \, (f_{\bullet}(a)+f_{\bullet}(b))[[r,\infty)],
\end{align*} that is, $f_{\bullet}(ab) \leq f_{\bullet}(a) + f_{\bullet}(b)$.

(iv) As $g^{-1}([0,1))$ is an escape neighborhood for $f$, there is $n \in \N$ with \begin{displaymath}
	\left\{ x \in G \left\vert \, g\!\left(x^{1}\right)\!,\ldots,g\!\left(x^{n}\right) < 1 \right\} \! \right. \, \subseteq \, f^{-1}\!\left(\left[0,\epsilon \right)\right).
\end{displaymath} 
Then, for each $c\in S(\mu, G)$, we have  \begin{align*}
	f_{\bullet}(c)[\left[ \epsilon,\infty \right) ] 
	&= \, \bigvee \left\{ c(x) \left\vert \, x \in G, \, f(x) \geq \epsilon \right\} \right. \\
	&\leq \, \bigvee\nolimits_{i=1}^{n} \bigvee \left\{ c(x) \left\vert \, x \in G, \, g\!\left(x^{i}\right) \geq 1 \right\} \right. \\
	&= \, \bigvee\nolimits_{i=1}^{n} \bigvee \{ c(x_{1})\wedge \ldots \wedge c(x_{i}) \mid x_{1},\ldots,x_{i} \in G, \, g(x_{1}\cdots x_{i}) \geq 1 \} \\
	&= \, \bigvee\nolimits_{i=1}^{n} \bigvee \left. \! \left\{ c^{i}(x) \, \right\vert x \in G, \, g(x) \geq 1 \right\}\\ 
		&=\, \bigvee\nolimits_{i=1}^{n} g_{\bullet}\!\left(c^{i}\right)\![[1,\infty)]. \qedhere
\end{align*}\end{proof}

\section{Homomorphism rigidity}\label{section:homomorphism.rigidity}

The current section is devoted to proving the theorem below, which is one of the two main results of the paper. 

\begin{thm}\label{theorem:escape} Let $G$ be a topological group, let $H$ be a topological group with the escape property, let $\phi$ be a pathological submeasure, and let $\mu$ be a measure. Then any continuous homomorphism from $L^{0}(\phi,G)$ to $L^{0}(\mu,H)$ is trivial. \end{thm}

Before proving Theorem~\ref{theorem:escape}, we derive from it a corollary characterizing pathological submeasures. If $\phi$ is a submeasure on a Boolean algebra $\mathcal{A}$, then we let $\mathcal{D}_{\phi}$ denote the topological group consisting of $\mathcal{A}$, equipped with the operation $\triangle \colon \mathcal{A} \times \mathcal{A} \to \mathcal{A}$ and the topology generated by the pseudometric $\mathcal{A} \times \mathcal{A} \to \R, \, (A,B) \mapsto \phi(A \mathbin{\triangle} B)$.

\begin{cor}\label{corollary:boolean.groups} Let $\phi$ be a submeasure. The following properties are equivalent. \begin{itemize}
	\item[(i)] $\phi$ is pathological.
	\item[(ii)] For every strictly positive measure $\mu$, any continuous homomorphism from $\mathcal{D}_{\phi}$ to $\mathcal{D}_{\mu}$ is trivial.
\end{itemize} \end{cor}

\begin{proof} (i)$\Longrightarrow$(ii). Suppose that $\phi$ is pathological. Let $\mu$ be a strictly positive measure. Then $\mathcal{D}_{\mu} \cong S(\mu,\Z_{2})$ is Hausdorff, therefore embeds into $L^{0}(\mu,\Z_{2})$ by Remark~\ref{remark:completion}. Moreover, $\mathcal{D}_{\phi} \cong S(\phi,\Z_{2})$. Hence, thanks to Proposition~\ref{proposition:extension.to.completion}(ii) and Theorem~\ref{theorem:escape}, any continuous homomorphism from $\mathcal{D}_{\phi}$ to $\mathcal{D}_{\mu}$ is trivial.
	
(ii)$\Longrightarrow$(i). Suppose that (ii) holds. Let $\phi$ be defined on a Boolean algebra $\mathcal{A}$. In order to deduce that $\phi$ is pathological, let $\mu$ be a measure on $\mathcal{A}$ such that $\mu \leq \phi$. According to Remark~\ref{remark:measure.zero.ideal}(ii), \begin{displaymath}
	\mu' \colon \, \mathcal{B} \, \longrightarrow \, \R, \quad A \mathbin{\triangle} \mathcal{N}_{\mu} \, \longmapsto \, \mu(A)
\end{displaymath} is a strictly positive measure on the Boolean algebra $\mathcal{B} \defeq \mathcal{A}/\mathcal{N}_{\mu}$. From $\mu \leq \phi$, we infer that the homomorphism $\pi \colon \mathcal{D}_{\phi} \to \mathcal{D}_{\mu'}, \, A \mapsto A \mathbin{\triangle} \mathcal{N}_{\mu}$ is continuous. Hence, (ii) implies that, for each $A \in \mathcal{A}$, we have $A \mathbin{\triangle} \mathcal{N}_{\mu} = \pi(A) = \mathcal{N}_{\mu}$, i.e., $\mu (A) = 0$. That is, $\mu = 0$. This shows that $\phi$ is pathological. \end{proof}

One of the key ideas in our approach to proving Theorem~\ref{theorem:escape} consists in drawing a link between continuous homomorphisms of $L^{0}$ groups on the one hand and certain maps connecting their underlying Boolean algebras on the other. We establish some convenient terminology.

\begin{definition} Let $\mathcal{A}$ and $\mathcal{B}$ be Boolean algebras. The \emph{$\vee$-monoid of $\mathcal{A}$} consists of the set $\mathcal{A}$ equipped with the operation $\vee \colon \mathcal{A} \times \mathcal{A} \to \mathcal{A}$ and the distinguished element $0 \in \mathcal{A}$. A \emph{$\vee$-monoid homomorphism} from $\mathcal{A}$ to $\mathcal{B}$ is a homomorphism from the $\vee$-monoid of $\mathcal{A}$ into that of $\mathcal{B}$, i.e., a map $\theta \colon \mathcal{A} \to \mathcal{B}$ such that $\theta (0) = 0$ and $\theta(A \vee B) = \theta (A) \vee \theta (B)$ for all $A,B \in \mathcal{A}$. Let $\phi$ be a submeasure on $\mathcal{A}$. A submeasure $\psi$ on $\mathcal{A}$ is said to be \emph{$\phi$-continuous} if \begin{displaymath}
	\forall \epsilon \in \R_{>0} \, \exists \delta \in \R_{>0} \, \forall A \in \mathcal{A} \colon \qquad \phi(A) \leq \delta \ \Longrightarrow \ \psi(A) \leq \epsilon . 
\end{displaymath} Suppose now that $\mu$ is a submeasure on $\mathcal{B}$. Then a $\vee$-monoid homomorphism $\theta \colon \mathcal{A} \to \mathcal{B}$ will be called \emph{$\phi$-to-$\mu$-continuous} if $\mu \circ \theta$ is $\phi$-continuous. \end{definition}

\begin{lem}\label{lemma:submodular} Let $\mathcal{A}$ and $\mathcal{B}$ be Boolean algebras, let $\theta \colon \mathcal{A} \to \mathcal{B}$ be a $\vee$-monoid homomorphism, and let $\mu$ be a submeasure on $\mathcal{B}$. The following hold. \begin{itemize}
	\item[(i)] $\mu \circ \theta$ is a submeasure on $\mathcal{A}$.
	\item[(ii)] If $\mu$ is submodular (e.g., a measure), then $\mu \circ \theta$ is submodular.
\end{itemize} \end{lem}

\begin{proof} Point (i) is clear. To see (ii), let $A, B \in \mathcal{A}$. Thanks to monotonicity of $\theta$, we have $\theta(A \wedge B) \leq\theta(A) \wedge \theta(B)$. Using submodularity and monotonicity of $\mu$, we infer \begin{align*}
	\mu(\theta(A \vee B)) & = \, \mu(\theta(A) \vee \theta(B)) \\
		& \leq \, \mu(\theta(A)) + \mu(\theta(B)) - \mu(\theta(A) \wedge \theta(B)) \\
		& \leq \, \mu(\theta(A)) + \mu(\theta(B)) - \mu(\theta(A \wedge B)). \qedhere
\end{align*} \end{proof}
	
\begin{lem}\label{lemma:christensen.kelley} Let $\phi$ be a pathological submeasure on a Boolean algebra $\mathcal{A}$ and $\mu$ be a submodular submeasure (e.g., a measure) on a Boolean algebra $\mathcal{B}$. If $\theta \colon \mathcal{A} \to \mathcal{B}$ is a $\phi$-to-$\mu$-continuous $\vee$-monoid homomorphism, then $\mu \circ \theta = 0$. \end{lem}
	
\begin{proof} By Lemma~\ref{lemma:submodular}, the map $\mu' \defeq \mu \circ \theta$ is a submodular submeasure on $\mathcal{A}$. Thanks to~\cite[Theorem~14]{kelley59}, thus there exists a measure $\nu \colon \mathcal{A} \to \R$ such that $\nu \leq \mu'$ and $\nu(1) = \mu'(1)$. As $\mu'$ is $\phi$-continuous, $\nu$ is $\phi$-continuous as well. Since $\phi$ is pathological and $\nu$ is a measure, a result of Christensen~\cite[Theorem~1(2)]{christensen} necessitates that $\nu = 0$. Hence, $0 = \nu(1) = \mu'(1)$. By monotonicity of $\mu'$, this entails that $\mu' = 0$. \end{proof}

As a final ingredient for the proof of Theorem~\ref{theorem:escape}, let us observe that Lemma~\ref{lemma:countable.integral} will allow us to restrict attention to groups of finite $\Z$-partitions of unity in a Boolean algebra: one advantage of these is their amenability as discrete groups; see~Remark~\ref{remark:inductive.limit}. Indeed, amenability will be used via F\o lner's criterion combined with the following basic observation.

\begin{lem}\label{lemma:folner} Let $G$ be a group. Let $A \subseteq G$ and $g \in G$ such that $A \cap gA = \emptyset$. Then, for all $F \in \Pfin(G)\setminus \{ \emptyset \}$ and $\epsilon \in \R_{>0}$, \begin{displaymath}
	\tfrac{\left\lvert F \mathbin{\triangle} gF \right\rvert}{\vert F \vert} \leq \epsilon \quad \Longrightarrow \quad \tfrac{\lvert F\setminus A \rvert}{\vert F \vert} \geq \tfrac{1-\epsilon}{2}.
\end{displaymath} \end{lem}

\begin{proof} Let $F \in \Pfin(G)\setminus \{ \emptyset \}$ and $\epsilon \in \R_{>0}$ be such that $\tfrac{\left\lvert F \mathbin{\triangle} gF \right\rvert}{\vert F \vert} \leq \epsilon$. Note that, for each $B\subseteq G$, we have \begin{displaymath}
	\bigl\lvert \lvert B\cap gF\rvert-\lvert B\cap F\rvert \bigr\rvert \, \leq \, \lvert F\mathbin{\triangle} gF\rvert \, \leq\, \epsilon \lvert F\rvert ,
\end{displaymath} so \begin{displaymath}
	\bigl\lvert \vert A \cap F \vert - \vert gA \cap F \vert \bigr\rvert \, = \, \bigl\lvert \vert gA \cap gF \vert - \vert gA \cap F \vert \bigr\rvert \, \leq \, \epsilon |F|. 
\end{displaymath} On the other hand, since $A \cap gA = \emptyset$, \begin{displaymath}
	\lvert A \cap F \rvert + \lvert gA \cap F \rvert \, = \, \lvert (A\cup gA) \cap F \rvert \, \leq \, \lvert F \rvert .
\end{displaymath} We conclude that \begin{displaymath}
	\lvert A \cap F \rvert \, \leq \, \lvert F \rvert - \lvert gA \cap F \rvert \, \leq \, \lvert F \rvert - \lvert A \cap F \rvert + \epsilon \lvert F \rvert
\end{displaymath} and hence $\tfrac{\lvert A \cap F \rvert}{\lvert F \rvert} \leq \tfrac{1+\epsilon}{2}$, that is, $\tfrac{\lvert F\setminus A \rvert}{\vert F \vert} \geq \tfrac{1-\epsilon}{2}$. \end{proof}

We arrive at a key lemma. Recall from Remark~\ref{remark:moore} that, if $(X,\mathcal{B},\mu)$ is a finite measure space, then $L^{0}(\mu,\R)$ is identified with $L^{0}(X,\mathcal{B},\mu; \R)$. Under this identification, we will view the elements of $L^{0}(\mu,\R)$ as measure equivalence classes of $\mu$-measurable functions from $X$ to $\R$. Note that the natural partial order on such equivalence classes extends the partial order on $S(\mu, \R)$ introduced just before Lemma~\ref{lemma:lifting.2}.

\begin{lem}\label{lemma:abstract.key} Let $\Gamma$ be a countable amenable group and let $(X,\mathcal{B},\mu)$ be a finite measure space. Let $\rho_{0},\rho_{1}\colon \Gamma \to L^{0}(\mu,\R)$ be such that \begin{displaymath}
	\forall a,b \in \Gamma \colon \qquad \rho_{0}\!\left(ab^{-1}\right)\! \, \leq \, \rho_{1}(a)+\rho_{1}(b) .
\end{displaymath} Then, for each $r \in \R$, we have 
\begin{displaymath}
\mu\!\left(\bigvee\nolimits_{a\in \Gamma} \rho_{0}(a)^{-1}([r,\infty))\right) \!\, \leq \, 4\sup\nolimits_{a\in \Gamma}\mu\!\left(\rho_{1}(a)^{-1}([r/2,\infty))\right).
\end{displaymath}  
\end{lem}

\begin{proof} Let $r \in \R$ and put $s \defeq \sup\nolimits_{a\in \Gamma}\mu\!\left(\rho_{1}(a)^{-1}([r/2,\infty))\right)$. Thanks to countability of $\Gamma$ and $\sigma$-additivity of $\mu$, it suffices to prove that \begin{equation}\label{finite.supports}
	\forall E \in \Pfin(\Gamma) \colon \qquad \mu\!\left(\bigvee\nolimits_{b \in E} \rho_{0}(b)^{-1}([r,\infty)) \right) \leq 4s .
\end{equation}

Let $E \in \Pfin(\Gamma)$. Let $\mathcal{Q}$ be the finite partition of $\bigvee\nolimits_{b \in E} \rho_{0}(b)^{-1}([r,\infty))$ in $\mathcal{B}/\mathcal{N}_{\mu}$ generated by $\left. \! \left\{ \rho_{0}(b)^{-1}([r,\infty)) \, \right\vert b \in E \right\}$. For each $Q \in \mathcal{Q}$, define \begin{displaymath}
	V_{Q} \, \defeq \, \left\{ a \in \Gamma \left\vert \, \mu\!\left(Q \wedge \rho_{1}(a)^{-1}([r/2,\infty)) \right) < \tfrac{1}{2}\mu(Q) \right\} . \right.
\end{displaymath} We argue that \begin{equation}\label{separation}
	\forall Q \in \mathcal{Q} \, \exists b \in E \colon \qquad V_{Q} \cap bV_{Q} = \emptyset .
\end{equation} To this end, let $Q \in \mathcal{Q}$. Then there exists $b \in E$ with $Q \leq \rho_{0}(b)^{-1}([r,\infty))$. Now, if $a \in V_{Q}$ and $ba \in V_{Q}$, then \begin{displaymath}
	\rho_{0}(b) \, = \, \rho_{0}\!\left(baa^{-1}\right)\! \, \leq \, \rho_{1}(ba) + \rho_{1}(a),
\end{displaymath} hence \begin{displaymath}
	\rho_{0}(b)^{-1}([r,\infty)) \, \leq \, \rho_{1}(ba)^{-1}([r/2,\infty)) \vee \rho_{1}(a)^{-1}([r/2,\infty))
\end{displaymath} and therefore \begin{align*}
	\mu(Q) \, &= \, \mu\!\left( Q \wedge \rho_{0}(b)^{-1}([r,\infty)) \right) \\
		& \leq \, \mu\!\left( Q \wedge \rho_{1}(ba)^{-1}([r/2,\infty)) \right) + \mu\!\left( Q \wedge \rho_{1}(a)^{-1}([r/2,\infty)) \right)\\ & < \, \tfrac{1}{2}\mu(Q)+\tfrac{1}{2} \mu(Q)\,=\,\mu(Q) ,
\end{align*} which is absurd. Thus, $V_{Q} \cap bV_{Q} = \emptyset$ as desired. This proves~\eqref{separation}.

Now, consider any $\tau \in (0,1)$. Since the (discrete) group $\Gamma$ is amenable, there exists $F \in \Pfin(\Gamma)\setminus \{ \emptyset \}$ such that \begin{displaymath}
	\forall b \in E \colon \qquad \tfrac{\vert F \mathbin{\triangle} bF \vert}{\vert F \vert} \leq \tau 
\end{displaymath} (see, e.g.,~\cite[Theorem~12.11(a)$\Leftrightarrow$(f)]{Wagon}). Thus, Lemma~\ref{lemma:folner} and~\eqref{separation} together give $\lvert F\setminus V_{Q} \rvert \geq \tfrac{1-\tau}{2} \vert F \vert$ for each $Q \in \mathcal{Q}$, which implies  \begin{align*}
	\tfrac{1-\tau}{4} \vert F \vert\, &\mu\!\left(\bigvee\nolimits_{b \in E} \rho_{0}(b)^{-1}([r,\infty)) \right) \, = \, \tfrac{1-\tau}{4} \vert F \vert \sum\nolimits_{Q \in \mathcal{Q}} \mu(Q) \\
		& \leq \, \sum\nolimits_{Q \in \mathcal{Q}} \tfrac{1}{2} \sum\nolimits_{a \in F\setminus V_{Q}} \mu(Q) \\
		& \leq \, \sum\nolimits_{Q \in \mathcal{Q}} \sum\nolimits_{a \in F} \mu\!\left(Q \wedge \rho_{1}(a)^{-1}([r/2,\infty))\right) \\
		&= \, \sum\nolimits_{a \in F} \mu\!\left( \left( \bigvee\nolimits_{b \in E} \rho_{0}(b)^{-1}([r,\infty)) \right) \wedge \rho_{1}(a)^{-1}([r/2,\infty)) \right) \\
		&\leq \, \sum\nolimits_{a \in F} \mu\!\left(\rho_{1}(a)^{-1}([r/2,\infty))\right) \, \leq \, s \vert F \vert ,
\end{align*} that is, $\mu\!\left(\bigvee\nolimits_{b \in E} \rho_{0}(b)^{-1}([r,\infty)) \right) \!\leq \tfrac{4s}{1-\tau}$. This shows that \begin{displaymath}
	\mu\!\left(\bigvee\nolimits_{b \in E} \rho_{0}(b)^{-1}([r,\infty)) \right)\! \, \leq \, 4s ,
\end{displaymath} finishing the proof of~\eqref{finite.supports} and thus completing the argument. \end{proof}

Turning to the application of Lemma~\ref{lemma:abstract.key} that is important to our considerations, let us recall from Lemma~\ref{lemma:lifting} that a continuous length function $f$ on a topological group $G$ induces a bilaterally uniformly continuous map \begin{displaymath}
	f_{\bullet} \colon \, L^0(\mu, G)\, \longrightarrow \, L^0(\mu, \R).  
\end{displaymath} The basic properties of $f_\bullet$ are listed in Lemma~\ref{lemma:lifting.2}. Once again, as suggested by Remark~\ref{remark:moore}, we will view the values taken by $f_\bullet$ as measure equivalence classes of $\mu$-measurable functions from $X$ to $\R$.

\begin{lem}\label{lemma:escape} Let $G$ be a countable amenable group, let $H$ be a topological group, let~$\phi$ be a submeasure on a countable Boolean algebra $\mathcal{A}$, let $(X,\mathcal{B},\mu)$ be a finite measure space, and let $\pi \colon S(\phi,G) \to L^{0}(\mu,H)$ be a continuous homomorphism. Furthermore, let $f$ be an escape function on $H$. Then\footnote{where the notation from Lemma~\ref{lemma:supported.subgroups} is used} \begin{displaymath}
	\theta \colon \, \mathcal{A} \, \longrightarrow \, \mathcal{B}/\mathcal{N}_{\mu}, \quad A \, \longmapsto \, \bigvee\nolimits_{a \in \Gamma(A)} f_{\bullet}(\pi(a))^{-1}((0,\infty))
\end{displaymath} is a $\phi$-to-$\mu$-continuous $\vee$-monoid homomorphism. If $\phi$ is pathological, then \begin{displaymath}
	\mu\!\left(\bigvee\nolimits_{a \in S(\phi,G)} f_{\bullet}(\pi(a))^{-1}((0,\infty)) \right)\! \, = \, \mu(\theta(1)) \, = \, 0 .
\end{displaymath} \end{lem}

\begin{proof} We start with a claim reformulating Lemma~\ref{lemma:lifting.2}(ii)--(iv) to the current context. 

\begin{claim} \begin{itemize}
	\item[\textnormal{(a)}] $f_{\bullet}\!\left( \mathcal{F}^{-1}\right) = f_{\bullet}(\mathcal{F})$, for every $\mathcal{F} \in L^{0}(\mu,H)$.
	\item[\textnormal{(b)}] $f_{\bullet}(\mathcal{F}\cdot \mathcal{G}) \leq f_{\bullet}(\mathcal{F}) + f_{\bullet}(\mathcal{G})$, for all $\mathcal{F},\mathcal{G} \in L^{0}(\mu,H)$.
	\item[\textnormal{(c)}] If $g \colon H \to [0,\infty)$ is bilaterally uniformly continuous and $g^{-1}([0,1))$ is an escape neighborhood for $f$, then \begin{displaymath}
					\qquad \forall \mathcal{F} \in L^{0}(\mu,H) \colon \quad f_{\bullet}(\mathcal{F})^{-1}((0,\infty)) \, \leq \, \bigvee\nolimits_{\mathcal{G} \in \langle \mathcal{F} \rangle} g_{\bullet}(\mathcal{G})^{-1}([1,\infty)) .
				\end{displaymath}
\end{itemize} \end{claim}

\noindent \emph{Proof of Claim.}
(a) and (b) are immediate consequences of points (ii) and (iii) in Lemma~\ref{lemma:lifting.2}, since $S(\mu,H)$ is dense in $L^0(\mu,H)$, 
multiplication and inversion are continuous, $f_\bullet$ is continuous, and the relation ${\leq}$ on $L^0(\mu, \R)$ constitutes a closed subset of 
$L^0(\mu, \R)\times L^0(\mu, \R)$. 

(c) Fix $\epsilon \in \R_{>0}$. Let $n$ be provided by Lemma~\ref{lemma:lifting.2}(iv). Then, by continuity of $f_\bullet$ and density of $S(\mu,H)$ in $L^0(\mu, H)$, we have that, for each $\mathcal{F} \in L^{0}(\mu,H)$,
\begin{displaymath}
	f_{\bullet}(\mathcal{F})^{-1}([\epsilon,\infty)) \, \leq \, \bigvee\nolimits_{1 \leq i \leq n} g_{\bullet}(\mathcal{F}^i)^{-1}([1,\infty)) .
\end{displaymath} It follows that \begin{displaymath}
	f_{\bullet}(\mathcal{F})^{-1}([\epsilon,\infty)) \, \leq \,  \bigvee\nolimits_{\mathcal{G} \in \langle \mathcal{F} \rangle} g_{\bullet}(\mathcal{G})^{-1}([1,\infty)) .
\end{displaymath} Since $\epsilon\in \R_{>0}$ was arbitrary, we get the conclusion of (c). \hfill $\qed_{\textbf{Claim}}$ 

\smallskip

The remainder of the proof will involve the notation introduced in Lemma~\ref{lemma:supported.subgroups}. First we show that $\theta$ is a $\vee$-monoid homomorphism. 
Evidently, $\theta(0) = 0$ by Lemma~\ref{lemma:supported.subgroups}(i). For all $A,B \in \mathcal{A}$, \begin{align*}
	\theta(A \vee B) \, &= \, \bigvee\nolimits_{c \in \Gamma(A\vee B)} f_{\bullet}(\pi(c))^{-1}((0,\infty)) \\
		& \stackrel{\ref{lemma:supported.subgroups}(ii)}{=} \, \bigvee\nolimits_{a \in \Gamma(A)} \bigvee\nolimits_{b \in \Gamma(B)} f_{\bullet}(\pi(ab))^{-1}((0,\infty)) \\
		& \stackrel{\pi\,\text{hom.}}{=} \bigvee\nolimits_{a \in \Gamma(A)} \bigvee\nolimits_{b \in \Gamma(B)} f_{\bullet}(\pi(a)\pi(b))^{-1}((0,\infty)) \\
		& \stackrel{\textnormal{(b)}}{=} \, \bigvee\nolimits_{a \in \Gamma(A)} \bigvee\nolimits_{b \in \Gamma(B)} f_{\bullet}(\pi(a))^{-1}((0,\infty)) \vee f_{\bullet}(\pi(b))^{-1}((0,\infty)) \\
		& = \, \left( \bigvee\nolimits_{a \in \Gamma(A)} f_{\bullet}(\pi(a))^{-1}((0,\infty)) \right) \vee \left( \bigvee\nolimits_{b \in \Gamma(B)} f_{\bullet}(\pi(b))^{-1}((0,\infty)) \right) \\
		& = \, \theta(A) \vee \theta (B) .
\end{align*} So $\theta$ is indeed a $\vee$-monoid homomorphism.
	
We proceed to the proof of $\phi$-to-$\mu$-continuity of $\theta$, which amounts to showing that, for each $\epsilon \in \R_{>0}$, there exists $\delta \in \R_{>0}$ such that \begin{equation}\label{E:dep}
	\forall A \in \mathcal{A} \colon \qquad \phi(A) \leq \delta \ \Longrightarrow \ \mu(\theta(A)) \leq \epsilon . 
\end{equation} Let $\epsilon \in \R_{>0}$. We fix an escape neighborhood $U \in \Neigh(H)$ for $f$ and choose some continuous length function $g$ on $H$ such that $g^{-1}([0,1)) \subseteq U$. Lemma~\ref{lemma:lifting} asserts that $g_{\bullet}$ is continuous. Furthermore, note that $g_{\bullet}(\pi(e)) = 0$. So, by continuity of $\pi$, there exists $\delta \in \R_{>0}$ such that \begin{equation}\label{lifting.continuity}
	\forall A \in \mathcal{A} \colon \quad \phi(A) \leq \delta \ \Longrightarrow \ \left( \forall a \in \Gamma(A) \colon \ \mu\!\left(g_{\bullet}(\pi(a))^{-1}([1/2,\infty))\right) \leq \tfrac{\epsilon}{4} \right) .
\end{equation} We will verify~\eqref{E:dep} for this choice of $\delta$. To this end, fix $A \in \mathcal{A}$ with $\phi(A) \leq \delta$. 

First, let us observe that $\Gamma(A)$ is countable and amenable as a discrete group by Remark~\ref{remark:inductive.limit}. Moreover, since $\pi$ is a homomorphism, we infer from (a) and (b) that the map $\rho_{0} = \rho_{1} \colon \Gamma(A) \to L^{0}(\mu,\R), \, a \mapsto g_{\bullet}(\pi(a))$ satisfies the hypothesis of Lemma~\ref{lemma:abstract.key}, which gives \begin{equation}\label{abstract.key.estimate}
	\mu\!\left(\bigvee\nolimits_{a\in \Gamma(A)} g_{\bullet}(\pi(a))^{-1}([1,\infty))\right) \!\, \leq \, 4\sup\nolimits_{a\in \Gamma(A)}\mu\!\left(g_{\bullet}(\pi(a))^{-1}([1/2,\infty))\right).
\end{equation} Since \begin{align*}
	\theta (A) \, &= \, \bigvee\nolimits_{b \in \Gamma(A)} f_{\bullet}(\pi(b))^{-1}((0,\infty)) \\
		& \stackrel{\textnormal{(c)}}{\leq} \, \bigvee\nolimits_{b \in \Gamma(A)} \bigvee\nolimits_{c \in \langle b \rangle} g_{\bullet}(\pi(c))^{-1}([1,\infty)) \, = \, \bigvee\nolimits_{b \in \Gamma(A)} g_{\bullet}(\pi(b))^{-1}([1,\infty)) ,
\end{align*} we conclude that \begin{align*}
	\mu(\theta(A)) \, &\leq \, \mu\!\left(\bigvee\nolimits_{b \in \Gamma(A)} g_{\bullet}(\pi(b))^{-1}([1,\infty)) \right)\\
	&\stackrel{\eqref{abstract.key.estimate}}{\leq} \, 4\sup\nolimits_{a\in \Gamma(A)}\mu\!\left(g_{\bullet}(\pi(a))^{-1}([1/2,\infty))\right) \! \, \stackrel{\eqref{lifting.continuity}}{\leq} \, \epsilon ,
\end{align*} as desired in~\eqref{E:dep}, thus proving $\phi$-to-$\mu$-continuity of $\theta$.
	
Finally, suppose that $\phi$ is pathological. Since $\theta \colon \mathcal{A} \to \mathcal{B}/\mathcal{N}_{\mu}$ is a $\phi$-to-$\mu$-continuous $\vee$-monoid homomorphism, Lemma~\ref{lemma:christensen.kelley} gives \begin{displaymath}
	\mu\!\left(\bigvee\nolimits_{a \in S(\phi,G)} f_{\bullet}(\pi(a))^{-1}((0,\infty)) \right)\! \, = \, \mu(\theta(1)) \, = \, 0 . \qedhere
\end{displaymath} \end{proof}

\begin{proof}[Proof of Theorem~\ref{theorem:escape}] We note the following consequence of Lemma~\ref{lemma:lifting.2}(i). 

\begin{claim} If $f$ is a continuous length function on $H$ and $\mathcal{F} \in L^{0}(\mu,H)$ with $f_{\bullet}(\mathcal{F}) = 0$, then \begin{displaymath}
	\qquad \forall \epsilon \in \R_{>0} \colon \qquad \mathcal{F} \, \in \, \overline{ N_{\mu}\!\left(f^{-1}([0,\epsilon]),\epsilon\right)}.
\end{displaymath} \end{claim}

\noindent \emph{Proof of Claim.} Let $\epsilon \in \R_{>0}$. Let $U$ be a neighborhood of $\mathcal{F}$ in $L^{0}(\mu,H)$. As $f_{\bullet}$ is continuous and $S(\mu,H)$ is dense in $L^{0}(\mu,H)$ by Remark~\ref{remark:completion}, there is $a \in S(\mu,H)$ with $a \in U$ and $\mu\!\left( f_{\bullet}(a)^{-1}([\epsilon,\infty)) \right)\! \leq \epsilon$. Then, by Lemma~\ref{lemma:lifting.2}(i), $a \in N_{\mu}\!\left(f^{-1}([0,\epsilon]),\epsilon\right)$, which proves that $U \cap N_{\mu}\!\left(f^{-1}([0,\epsilon]),\epsilon\right) \ne \emptyset$, as required. \hfill $\qed_{\textbf{Claim}}$

\smallskip

Since $L^{0}(\mu,H)$ is Hausdorff and $S(\phi,G)$ is dense in $L^{0}(\phi,G)$ by Remark~\ref{remark:completion}, it suffices to prove that any continuous homomorphism from $S(\phi,G)$ to $L^{0}(\mu,H)$ is trivial. By Lemma~\ref{lemma:countable.integral}, it will be enough to verify this for the discrete group $G = \Z$ and $\phi$ being defined on a countable Boolean algebra~$\mathcal{A}$, as we will do. Let $\mu$ be defined on a Boolean algebra $\mathcal{B}$. By Proposition~\ref{proposition:loomis.sikorski}, there exist a finite measure space $(X,\mathcal{B}',\mu')$ and a Boolean algebra homomorphism $\beta \colon \mathcal{B} \to \mathcal{B}'/\mathcal{N}_{\mu'}$ such that $\mu(B) = \mu'(\beta(B))$ for all $B \in \mathcal{B}$. Due to Remark~\ref{remark:boolean.transformations}, the topological group $L^{0}(\mu,H)$ embeds into $L^{0}(\mu',H)$. It thus remains to verify that any continuous homomorphism from $S(\phi,\Z)$ to $L^{0}(\mu',H)$ is trivial. So, let $\pi \colon S(\phi,\Z) \to L^{0}(\mu',H)$ be a continuous homomorphism. If $U \in \Neigh(H)$ and $\epsilon \in \R_{>0}$, then the escape property of $H$ produces an escape function $f$ on $H$ satisfying $f^{-1}([0,1]) \subseteq U$, while Lemma~\ref{lemma:escape} asserts that \begin{displaymath}
	\mu'\!\left(\bigvee\nolimits_{a \in S(\phi,G)} f_{\bullet}(\pi(a))^{-1}((0,\infty)) \right)\! \, = \, 0;
\end{displaymath} consequently, for every $a \in S(\phi,\Z)$, one has $f_{\bullet}(\pi(a)) = 0$ and thus  \begin{displaymath}
	\pi(a) \, \in \, \overline{N_{\mu'}\!\left(f^{-1}([0,1]),\epsilon\right)} \, \subseteq \, \overline{N_{\mu'}(U,\epsilon)} ,
\end{displaymath} by the claim established at the beginning of the proof. As $L^{0}(\mu',H)$ is Hausdorff and \begin{displaymath}
	\left. \left\{ \overline{N_{\mu'}(U,\epsilon)} \, \right\vert U \in \Neigh(H), \, \epsilon \in \R_{>0} \right\}
\end{displaymath} is a neighborhood basis at the identity in $L^{0}(\mu',H)$, it follows that $\pi$ is trivial. 
\end{proof}

\section{Exoticness}\label{section:exoticness}

A topological group $G$ is called \emph{exotic}~\cite{HererChristensen} if $G$ does not admit any non-trivial strongly continuous unitary representation on a Hilbert space, i.e., any homomorphism from~$G$ into the unitary group $\U(H)$ of a Hilbert space $H$ that is continuous with respect to the strong operator topology on $\U(H)$ must be constant. A topological group $G$ is said to be \emph{strongly exotic}~\cite{banaszczyk} if $G$ does not admit any non-trivial weakly continuous representation on a Hilbert space, i.e., any homomorphism from $G$ into the group~$\GL(H)$ of invertible bounded linear operators on a Hilbert space $H$, continuous with respect to the weak operator topology on $\GL(H)$, is constant. The following closure property of the class of (strongly) exotic topological groups will be useful.

\begin{remark}\label{remark:strongly.exotic} Let $G$ and $H$ be topological groups and let $\pi \colon G \to H$ be a continuous homomorphism such that $\overline{\pi(G)} = H$. If $G$ is exotic (resp., strongly exotic), then $H$ is exotic (resp., strongly exotic). This is an immediate consequence of the fact that the weak operator topology (and thus the strong operator topology) on the set of bounded linear endomorphisms of a Hilbert space is Hausdorff. \end{remark}

Evidently, strong exoticness implies exoticness. The converse implication holds under some additional hypotheses. A topological group $G$ is called \emph{bounded (in the sense of Bourbaki)} (cf.~\cite[II, \S4, Exercise 7, p.~210]{bourbaki}) if, for every $U \in \Neigh (G)$, there exist $F \in \Pfin(G)$ and $n \in \N$ such that $U^{n}F = G$. We recall that a topological group is bounded if and only if every left-uniformly continuous real-valued function on it is bounded (see~\cite[Theorem~1.4]{hejcman} or~\cite[Theorem~2.4]{atkin}).

\begin{lem}[Banaszczyk~\cite{banaszczyk}\footnote{This follows from~\cite[Lemma~3]{banaszczyk} and~\cite[Lemma~4]{banaszczyk} (see also~\cite[Proposition~3]{banaszczyk2}).}]\label{lemma:banaszczyk} Let $G$ be a topological group. Suppose that $G$ is \begin{itemize}
	\item[(i)] first-countable,
	\item[(ii)] bounded,
	\item[(iii)] and amenable as a discrete group.
\end{itemize} If $G$ is exotic, then $G$ is strongly exotic. \end{lem}

\begin{proof} We include a proof for the sake of convenience. Let $\pi \colon G \to \GL(H)$ be a weakly continuous representation on some Hilbert space $H$. As $G$ is first-countable, there exists a decreasing sequence $(U_{n})_{n \in \N} \in \Neigh (G)^{\N}$ such that $\{ U_{n} \mid n \in \N \}$ is a neighborhood basis at the neutral element in $G$. For every $(g_{n})_{n \in \N} \in \prod_{n \in \N} U_{n}$, the sequence $(\pi(g_{n}))_{n \in \N}$ converges to $\id_{H}$ with respect to the weak operator topology, hence $\sup \{ \Vert \pi(g_{n}) \Vert \mid n \in \N \} < \infty$ by two applications of the Banach--Steinhaus theorem. It follows that $\sup \{ \Vert \pi(g) \Vert \mid g \in U_{n} \} < \infty$ for some $n \in \N$. Since $G$ is bounded, there exist $F \in \Pfin(G)$ and $k \in \N$ with $G = U_{n}^{k}F$. We conclude that \begin{displaymath}
	\forall g \in G \colon \qquad \Vert \pi(g) \Vert \, \leq \, \left( \sup\nolimits_{u \in U_{n}} \Vert \pi(u) \Vert \right)^{k}\sup\nolimits_{h \in F} \Vert \pi(h) \Vert	.
\end{displaymath} Consequently, $C \defeq \sup\nolimits_{g \in G} \Vert \pi(g) \Vert < \infty$. 

We proceed to applying an argument by Day~\cite[Theorem~8]{day} and Dixmier~\cite[Théorème 6]{dixmier} (see also~\cite[Lemma~6.1.3]{gheysens}). Consider the family of bounded functions \begin{displaymath}
	f_{x,y} \colon \, G \, \longrightarrow \, \C, \quad g \, \longmapsto \, \langle \pi(g)x, \pi(g)y \rangle \qquad (x,y \in H)
\end{displaymath} and note that \begin{displaymath}
	\forall x \in H \, \forall g \in G \colon \qquad C^{-2}\langle x,x \rangle \, \leq \, f_{x,x}(g) \, \leq \, C^{2}\langle x,x \rangle .
\end{displaymath} Since $G$ is amenable as a discrete group, the algebra $\ell^{\infty}(G)$ of all bounded complex-valued functions on $G$ admits a right-invariant mean, i.e., there exists a positive unital linear map $\mu \colon \ell^{\infty}(G) \to \C$ such that \begin{displaymath}
	\forall g \in G \ \forall f \in \ell^{\infty}(G) \colon \quad \mu(f \circ {\rho_{g}}) \, = \, \mu(f) ,
\end{displaymath} where $\rho_{g} \colon G \to G, \, x \mapsto xg$ for $g \in G$. One readily checks that the map \begin{displaymath}
	H \times H \, \longrightarrow \, \C, \quad (x,y) \, \longmapsto \, \langle x,y \rangle_{\mu} \defeq \mu(f_{x,y})
\end{displaymath} constitutes an inner product satisfying \begin{equation}\label{sandwich}
	\forall x \in H \colon \qquad C^{-2}\langle x,x \rangle \, \leq \,\langle x,x \rangle_{\mu} \, \leq \, C^{2}\langle x,x \rangle .
\end{equation} In turn, we obtain a Hilbert space $H_{\mu}$ by equipping the underlying complex vector space of $H$ with the inner product $\langle {\cdot}, {\cdot} \rangle_{\mu}$. It follows from~\eqref{sandwich} that the norm topologies of $H$ and $H_{\mu}$ coincide, and so do the corresponding weak topologies. Thus, the induced weak operator topologies on $\GL(H) = \GL(H_{\mu})$ agree. Furthermore, \begin{displaymath}
	\langle \pi(g)x,\pi(g)y \rangle_{\mu} \, = \, \mu(f_{\pi(g)x,\pi(g)y}) \, = \, \mu({f_{x,y}} \circ {\rho_{g}}) \, = \, \mu(f_{x,y}) \, = \, \langle x,y \rangle_{\mu} 
\end{displaymath} for all $g \in G$ and $x,y \in H$, that is, $\pi$ is a unitary representation of $G$ on $H_{\mu}$. Our topological observations imply that $\pi \colon G \to \U(H_{\mu})$ is weakly continuous, therefore strongly continuous by~\cite[Section~9.1, Corollary~9.4, p.~256]{HilgertNeeb}. Hence, if $G$ is exotic, then $\pi$ is constant. \end{proof}

In preparation of the proof of the subsequent Lemma~\ref{lemma:exotic}, let us recall that a measure space $(X,\mathcal{B},\mu)$ is said to be \emph{localizable} if \begin{itemize}
	\item[---] $\mu$ is \emph{semifinite}, i.e., for every $B \in \mathcal{B}$ with $\mu(B) > 0$ there exists $A \in \mathcal{B}$ such that $A \subseteq B$ and $0 < \mu(A) < \infty$,
	\item[---] and the Boolean algebra $\mathcal{B}/\mathcal{N}_{\mu}$ is complete.
\end{itemize} A measure space $(X,\mathcal{B},\mu)$ is localizable if and only if the linear operator \begin{displaymath}
	L^{\infty}(\mu,\C) \, \longrightarrow \, L^{1}(\mu,\C)^{\ast}, \quad f \, \longmapsto \, \left[ g \mapsto \int fg \, \mathrm{d}\mu \right]
\end{displaymath} is a (necessarily isometric) isomorphism of Banach spaces (see~\cite[Theorem~243G]{Fremlin2}). Furthermore, given a measure space $(X,\mathcal{B},\mu)$ and some $A \in \mathcal{B}$, we will consider the induced measure space $(A,\mathcal{B}_{A},\mu_{A})$ where $\mathcal{B}_{A} \defeq \{ B \cap A \mid B \in \mathcal{B} \}$ and $\mu_{A} \defeq \mu\vert_{\mathcal{B}_{A}}$.

\begin{remark}\label{remark:unitary.group} If $M$ is a von Neumann algebra acting on a Hilbert space $H$, then the unitary group $\U(M) \defeq \{ u \in M \mid uu^{\ast} = u^{\ast}u = 1 \}$, endowed with the strong operator topology inherited from $M \subseteq \B(H)$ is a topological group. If $(X,\mathcal{B},\mu)$ is a finite measure space, then $N \defeq L^{\infty}(X,\mathcal{B},\mu;\C)$ is a commutative von Neumann algebra acting on the Hilbert space $L^{2}(X,\mathcal{B},\mu;\C)$, and the strong operator topology on $\U(N) = L^{0}(X,\mathcal{B},\mu;\T)$ inherited from $N$ coincides with the topology of convergence in the measure $\mu$. \end{remark}

\begin{lem}\label{lemma:exotic} If $\phi$ is a pathological submeasure, then $S(\phi,\Z)$ is strongly exotic. \end{lem}

\begin{proof} Let $\phi$ be a pathological submeasure. As $G \defeq S(\phi,\Z)$ is first-countable by Remark~\ref{remark:L0.metrizable}, bounded by Corollary~\ref{corollary:diffuse} and Lemma~\ref{lemma:power.bounded}, and abelian, thus amenable as a discrete group, Lemma~\ref{lemma:banaszczyk} asserts that it suffices to prove that $G$ is exotic. So, let $H$ be a Hilbert space and let $\pi \colon G \to \U(H)$ be a strongly continuous homomorphism. Since the von Neumann algebra $M$ acting on $H$ generated by $\pi(G)$ is abelian, the classification of abelian von Neumann algebras (see, e.g.,~\cite[Proposition~1.18.1, p.~45]{sakai} or \cite[Corollary~7.20]{BlecherGoldsteinLabuschagne}) asserts the existence of a localizable measure space $(X,\mathcal{B},\mu)$ and a ${}^{\ast}$-isomorphism $\iota \colon M \to N$ onto the von Neumann algebra $N \defeq L^{\infty}(X,\mathcal{B},\mu;\C)$ acting on the Hilbert space $K \defeq L^{2}(X,\mathcal{B},\mu;\C)$. Since $\iota$ is necessarily unital, we have $\iota(\U(M)) = \U(N)$, hence $\pi' \defeq \iota \circ \pi \colon G \to \U(N)$ is a well-defined homomorphism. We argue that $\pi'$ is strongly continuous. First, let us note that $\iota$ is continuous with regard to the respective ultraweak topologies on $M$ and $N$ (see~\cite[Corollary~5.13, p.~120]{LecturesOnVonNeumannAlgebras}). On the norm-bounded subset $\U(M)$ (resp., $\U(N)$), the ultraweak topology coincides with the weak operator topology (see~\cite[Item~(iv) of Lemma before Theorem~1.4, p.~7]{LecturesOnVonNeumannAlgebras}), and since $\U(M) \subseteq \U(H)$ (resp., $\U(N) \subseteq \U(K)$), the latter agrees with the strong operator topology (see~\cite[Section~9.1, Corollary~9.4, p.~256]{HilgertNeeb}). It follows that $\pi'$ is indeed strongly continuous. Now, for each $A \in \mathcal{B}$ with $\mu(A) < \infty$, we observe that $N_{A} \defeq L^{\infty}(A,\mathcal{B}_{A},\mu_{A};\C)$ is a von Neumann algebra acting on the Hilbert space $K_{A} \defeq L^{2}(A,\mathcal{B}_{A},\mu_{A};\C)$ and the mapping $\theta_{A} \colon N \to N_{A}, \, f \mapsto {f\vert_{A}}$ is a strongly continuous unital ${}^{\ast}$-homomorphism, thus induces a strongly continuous homomorphism from $\U(N)$ to \begin{displaymath}
	\U(N_{A}) \, \stackrel{\ref{remark:unitary.group}}{=} \, L^{0}(A,\mathcal{B}_{A},\mu_{A};\T) \, \stackrel{\ref{remark:moore}}{\cong} \, L^{0}(\mu_{A},\T) ,
\end{displaymath} so that ${\theta_{A}} \circ {\pi'}$ is constant by Proposition~\ref{proposition:escape.groups} and Theorem~\ref{theorem:escape} (or by Theorem~\ref{theorem:alternative}). Since $\mu$ is semifinite, this entails that $\pi'$ is constant, so $\pi$ must be constant, too. \end{proof}

\begin{thm}\label{theorem:exotic} If $\phi$ is a pathological submeasure and $G$ is a topological group, then both $S(\phi,G)$ and $L^{0}(\phi,G)$ are strongly exotic. \end{thm}

\begin{proof} Let $\phi$ be a pathological submeasure and $G$ be a topological group. Then $S(\phi,G)$ is strongly exotic according to Lemma~\ref{lemma:exotic} and Lemma~\ref{lemma:countable.integral}. Since $S(\phi, G)$ is dense in $L^{0}(\phi,G)$ by Remark~\ref{remark:completion}, thus $L^{0}(\phi,G)$ is strongly exotic by Remark~\ref{remark:strongly.exotic}. \end{proof}

\begin{cor}\label{corollary:exotic.embedding} Every topological group embeds into a strongly exotic one. \end{cor}

\begin{proof} This is a consequence of Theorem~\ref{theorem:exotic}, combined with Remark~\ref{remark:homomorphism} and the existence of non-zero pathological submeasures~\cite[Corollary to Theorem~2, p.~205]{HererChristensen} (see also~\cite{talagrand80} for a family of simple examples). \end{proof}

The fact that every second-countable topological group embeds into an exotic one follows already from~\cite[Corollary~1.4]{Pestov07} and~\cite{usp}.

\begin{remark}\label{remark:whirly} Inspired by~\cite{GlasnerTsirelsonWeiss,GlasnerWeiss,pestov10,PestovSchneider}, we call a topological group $G$ \emph{whirly} if, for every continuous action of $G$ on a compact Hausdorff space $X$ and every $G$-invariant regular Borel probability $\mu$ on $X$, the support of $\mu$ is contained in the set of $G$-fixed points in $X$, that is, the induced strongly continuous unitary representation of $G$ on $L^{2}(\mu)$ is trivial. So, exoticness implies whirlyness. These two properties are not, however, equivalent even within the class of amenable topological groups---for instance, the unitary group of an infinite-dimensional Hilbert space is both amenable and whirly with respect to the strong operator topology~\cite[Theorem~1.1]{GlasnerTsirelsonWeiss}, but is not exotic. On the other hand, the conjunction of whirlyness and amenability implies extreme amenability (as the support of a regular Borel probability measure is never empty). In particular, exotic amenable topological groups are extremely amenable. An example of an extremely amenable, non-whirly topological group is given by the automorphism group $\Aut(\Q,{\leq})$ of the linearly ordered set of rational numbers, equipped with the topology of pointwise convergence, due to~\cite[Theorem~5.4]{pestov98} and~\cite[Remark~1.3]{GlasnerTsirelsonWeiss}. \end{remark}

\begin{cor}[cf.~\cite{SchneiderSolecki} Corollary~7.6]\label{corollary:extreme.amenability} Let $\phi$ be a pathological submeasure, and let $G$ be a topological group. Then $L^{0}(\phi,G)$ is whirly. In particular, if $G$ is amenable, then $L^{0}(\phi,G)$ is extremely amenable. \end{cor}

\begin{proof} Combining Theorem~\ref{theorem:exotic} with Remark~\ref{remark:whirly}, we see that $L^{0}(\phi,G)$ is whirly. If $G$ is amenable, then an argument along the lines of Remark~\ref{remark:inductive.limit}, using closure properties of the class of amenable topological groups~\cite[Corollary~4.5, Theorems~4.6, 4.7, 4.8]{rickert}, yields amenability of the topological group $L^{0}(\phi,G)$, whence $L^{0}(\phi,G)$ is extremely amenable thanks to Remark~\ref{remark:whirly}. \end{proof}

The conclusion of Corollary~\ref{corollary:extreme.amenability} is established in~\cite{SchneiderSolecki} for a larger class of submeasures: if $\phi$ is a \emph{non-elliptic}\footnote{See~\cite[Definition~4.6]{SchneiderSolecki}. All pathological submeasures and all diffuse measures are non-elliptic.} diffuse submeasure, then $L^{0}(\phi,\Z)$ is whirly by~\cite[Corollary~7.6]{SchneiderSolecki}, so that an argument as in the proof of Lemma~\ref{lemma:countable.integral} shows that $L^{0}(\phi,G)$ is whirly for any topological group $G$. In turn, whirlyness of the associated $L^{0}$ groups is not a distinguishing characteristic of pathological submeasures. Exoticness, on the other hand, is such: Theorem~\ref{theorem:exotic} extends to a characterization of the class of pathological submeasures in terms of exoticness; see Corollary~\ref{corollary:exotic} and Corollary~\ref{corollary:exotic.2}. 

In order to deduce further corollaries of Theorem~\ref{theorem:exotic}, we briefly recall some well-known basic facts concerning functions of positive type. For further details, the reader is referred to~\cite[Appendix~C.4]{BekkaEtAlBook}.

\begin{definition} Let $G$ be a group. A function $f \colon G \to \C$ is said to be \emph{of positive type} if, for all $n \in \N$, $c_{1},\ldots,c_{n} \in \C$ and $g_{1},\ldots,g_{n} \in G$, \begin{displaymath}
	\sum\nolimits_{i,j=1}^{n} c_{i}\overline{c_{j}} f\!\left(g_{j}^{-1}g_{i}\right)\! \, \geq \, 0 .
\end{displaymath} \end{definition}

\begin{lem}[\cite{BekkaEtAlBook}, Proposition~C.4.2(ii),\,(iii)]\label{lemma:positive.type} Let $G$ be a group and let $f \colon G \to \C$ be of positive type. Then, for all $g,h \in G$, \begin{itemize}
	\item[(i)] $\vert f(g) \vert \leq f(e)$, 
	\item[(ii)] $\vert f(g) - f(h) \vert^{2} \leq 2f(e)\!\left(f(e)-\Re f\!\left(g^{-1}h\right)\right)$.
\end{itemize} \end{lem}

Exoticness of topological groups admits a well-known characterization via continuous functions of positive type, as spelled out in Proposition~\ref{proposition:positive.type}(i) below. We take the opportunity to recall a related characterization of quite a different property: a topological group $G$ is said to be \emph{unitarily representable} if there exists a Hilbert space $H$ such that $G$ is isomorphic to a topological subgroup of $\U(H)$ endowed with the strong operator topology. The class of unitarily representable topological groups includes all locally compact as well as all non-archimedean topological groups. Of course, unitary representability is inherited by topological subgroups. In the converse direction, let us note the following for later use.

\begin{remark}\label{remark:unitarily.representable} If $G$ is a Hausdorff topological group containing a dense unitarily representable topological subgroup, then $G$ itself must be unitarily representable, as follows from Remark~\ref{remark:isometry.groups.complete}(i) and Proposition~\ref{proposition:extension.to.completion}(iii). \end{remark}

\begin{prop}\label{proposition:positive.type} A topological group $G$ is \begin{itemize}
	\item[(i)] exotic if and only if every continuous complex-valued function on $G$ of positive type is constant,
	\item[(ii)] unitarily representable if and only if $G$ is Hausdorff and, for each $U \in \Neigh (G)$, there exist $\epsilon \in \R_{>0}$ and a continuous function $f \colon G \to \C$ of positive type with \begin{displaymath}
			\qquad \{ g \in G \mid \vert f(g) \vert \geq (1-\epsilon)f(e)\} \, \subseteq \, U .
		\end{displaymath}
\end{itemize} \end{prop}

\begin{proof} In both statements, the implication ($\Longrightarrow$) follows from~\cite[Theorem~C.4.10]{BekkaEtAlBook}, while ($\Longleftarrow$) is a consequence of~\cite[Proposition~C.4.3]{BekkaEtAlBook}. (Additionally, a proof of~(ii) is to be found in both~\cite[Lemma~2.1]{galindo} and~\cite[Theorem~2.1]{gao}.)\end{proof}

We return to the study of topological $L^{0}$ groups.

\begin{lem}\label{lemma:representation} Let $G$ be a topological group and let $\mu$ be a non-zero measure. \begin{itemize}
	\item[(i)] If $f \colon G \to \C$ is a continuous function of positive type, then \begin{displaymath}
					\qquad f' \colon \, S(\mu,G) \, \longrightarrow \, \C, \quad a \, \longmapsto \, \tfrac{1}{\mu(1)}\sum\nolimits_{g \in G} f(g)\mu(a(g))
				\end{displaymath} is a continuous function of positive type and $f = {f'} \circ {\eta_{\mu,G}}$.
	\item[(ii)] If $S(\mu,G)$ is exotic, then so is $G$.
	\item[(iii)] If $G$ is unitarily representable, then so is $S(\mu,G)/\,\overline{\!\{e_{S(\mu,G)}\}\!}\,$.
\end{itemize} \end{lem}

\begin{proof} (i) Let $f \colon G \to \C$ be a continuous function of positive type, and let $f'$ be defined as above. First of all, \begin{equation}\label{fundamental}
	f'(\eta_{\mu,G}(g)) \, = \, \tfrac{1}{\mu(1)} \sum\nolimits_{h \in G} f(h)\mu(\eta_{\mu,G}(g)(h)) \, = \, \tfrac{1}{\mu(1)}f(g)\mu(1) \, = \, f(g)
\end{equation} for every $g \in G$, that is, $f = {f'} \circ {\eta_{\mu,G}}$. We proceed to proving that $f'$ is of positive type. To this end, let $n \in \N$, $c_{1},\ldots,c_{n} \in \C$ and $a_{1},\ldots,a_{n} \in S(\mu,G)$. Since $f$ is of positive type, \begin{equation}\label{positive.type}
	\forall g_{1},\ldots,g_{n} \in G \colon \qquad \sum\nolimits_{i,j=1}^{n} c_{i}\overline{c_{j}} f\!\left(g_{j}^{-1}g_{i}\right)\! \, \geq \, 0 .
\end{equation} Furthermore, for all $g \in G$ and $i,j \in \{ 1,\ldots,n\}$, \begin{align*}
	\left( a_{j}^{-1}a_{i} \right)\!(g) \, &= \, \bigveedot \left. \! \left\{ a_{j}^{-1}(x) \wedge a_{i}(y) \, \right\vert x,y \in G, \, xy= g \right\} \\
	& = \, \bigveedot \left. \! \left\{ a_{j}\!\left(x^{-1}\right) \! \wedge a_{i}(y) \, \right\vert x,y \in G, \, xy= g \right\} \\
	& = \, \bigveedot \left\{ a_{j}(x) \wedge a_{i}(y) \left\vert \, x,y \in G, \, x^{-1}y= g \right\} \right. \\
	& = \, \bigveedot \left\{ a_{1}(g_{1}) \wedge \ldots \wedge a_{n}(g_{n}) \left\vert \, g_{1},\ldots,g_{n} \in G, \, g_{j}^{-1}g_{i}= g \right\} \right.
\end{align*} and thus \begin{equation}\label{partitioning}
	\mu\!\left( \!\left( a_{j}^{-1}a_{i} \right)\!(g) \right)\! \, \stackrel{\mu\,\text{meas.}}{=} \, \sum\nolimits_{g_{1},\ldots,g_{n} \in G, \, g_{j}^{-1}\!g_{i} = g} \mu(a_{1}(g_{1}) \wedge \ldots \wedge a_{n}(g_{n})) .
\end{equation} We conclude that \begin{align*}
	&\sum\nolimits_{i,j=1}^{n} c_{i}\overline{c_{j}} f'\!\left( a_{j}^{-1}a_{i}\right) \! \, = \, \tfrac{1}{\mu(1)} \sum\nolimits_{i,j=1}^{n} c_{i}\overline{c_{j}} \sum\nolimits_{g \in G} f(g)\mu\!\left( \! \left( a_{j}^{-1}a_{i} \right)\!(g) \right) \\
	&\quad \stackrel{\eqref{partitioning}}{=} \, \tfrac{1}{\mu(1)} \sum\nolimits_{i,j=1}^{n} c_{i}\overline{c_{j}} \sum\nolimits_{g \in G} f(g) \sum\nolimits_{g_{1},\ldots,g_{n} \in G, \, g_{j}^{-1}\!g_{i} = g} \mu(a_{1}(g_{1}) \wedge \ldots \wedge a_{n}(g_{n})) \\
	&\quad = \, \tfrac{1}{\mu(1)} \sum\nolimits_{i,j=1}^{n} c_{i}\overline{c_{j}} \sum\nolimits_{g_{1},\ldots,g_{n} \in G} f\!\left( g_{j}^{-1}g_{i}\right)\!\mu(a_{1}(g_{1}) \wedge \ldots \wedge a_{n}(g_{n})) \\
	&\quad = \, \tfrac{1}{\mu(1)} \sum\nolimits_{g_{1},\ldots,g_{n} \in G} \left( \sum\nolimits_{i,j=1}^{n} c_{i}\overline{c_{j}} f\!\left( g_{j}^{-1}g_{i}\right) \!\right) \!\mu(a_{1}(g_{1}) \wedge \ldots \wedge a_{n}(g_{n})) \, \stackrel{\eqref{positive.type}}{\geq} \, 0 .
\end{align*} This shows that $f'$ is indeed of positive type. It remains to prove continuity of $f'$. As $f'$ is of positive type, Lemma~\ref{lemma:positive.type}(ii) asserts that it suffices to verify that $f'$ is continuous at $e_{S(\mu,G)}$. For this purpose, let $\epsilon \in \R_{>0}$. Since $f$ is continuous, there exists $U \in \Neigh (G)$ such that $\vert f(e)-f(g) \vert \leq \tfrac{\epsilon}{2}$ for every $g \in U$. Consider \begin{displaymath}
	V \, \defeq \, N_{\mu}\!\left( U, \tfrac{\epsilon \mu(1)}{4\max\{ f(e),1\}}\right)\! \, \in \, \Neigh (S(\mu,G)) .
\end{displaymath} Now, if $a \in V$, then from \begin{displaymath}
	f'(e_{S(\mu,G)}) \, \stackrel{\eqref{fundamental}}{=} \, f(e) \, = \, \tfrac{1}{\mu(1)}f(e)\mu\!\left(\bigveedot\nolimits_{g \in G} a(g) \right)\! \, \stackrel{\mu\,\text{meas.}}{=} \, \tfrac{1}{\mu(1)} \sum\nolimits_{g \in G} f(e) \mu(a(g))
\end{displaymath} we infer that \begin{align*}
	\lvert f'(e_{S(\mu,G)}) - f'(a) \rvert \, &\leq \, \tfrac{1}{\mu(1)}\sum\nolimits_{g \in G} \lvert f(e)-f(g) \rvert \mu(a(g)) \\
	&\hspace{-25mm} = \, \tfrac{1}{\mu(1)}\sum\nolimits_{g \in U} \lvert f(e)-f(g) \rvert \mu(a(g)) + \tfrac{1}{\mu(1)}\sum\nolimits_{g \in G\setminus U} \lvert f(e)-f(g) \rvert \mu(a(g)) \\
	&\hspace{-25mm} \stackrel{\ref{lemma:positive.type}(i)}{\leq} \, \tfrac{\epsilon}{2\mu(1)}\sum\nolimits_{g \in U} \mu(a(g)) + \tfrac{2f(e)}{\mu(1)}\sum\nolimits_{g \in G\setminus U} \mu(a(g)) \\
	&\hspace{-25mm} \stackrel{\mu\,\text{meas.}}{=} \, \tfrac{\epsilon}{2\mu(1)} \mu\!\left(\bigveedot\nolimits_{g \in U} a(g)\right) + \tfrac{2f(e)}{\mu(1)}\mu\!\left(\bigveedot\nolimits_{g \in G \setminus U} a(g)\right)\! \, \leq \, \tfrac{\epsilon}{2} + \tfrac{\epsilon}{2} \, = \, \epsilon .
\end{align*} This shows that $f'$ is continuous at $e_{S(\mu,G)}$, hence continuous by Lemma~\ref{lemma:positive.type}(ii).
	
(ii) Suppose that $S(\mu,G)$ is exotic. If $f \colon G \to \C$ is a continuous function of positive type, then~(i) asserts the existence of a continuous function $f' \colon S(\mu,G) \to \C$ of positive type with $f = {f'} \circ {\eta_{\mu,G}}$, but Proposition~\ref{proposition:positive.type}(i) implies that $f'$ is constant, so $f = {f'} \circ {\eta_{\mu,G}}$ is constant. By Proposition~\ref{proposition:positive.type}(i), this shows that $G$ is exotic.

(iii) Assume that $G$ is unitarily representable. We deduce unitary representability of $S(\mu,G)/\,\overline{\!\{e_{S(\mu,G)}\}\!}\,$ using Proposition~\ref{proposition:positive.type}(ii). To this end, let $U \in \Neigh (S(\mu,G))$. Then there exist $V \in \Neigh(G)$ and $\epsilon \in \R_{>0}$ such that $N_{\mu}(V,\epsilon) \subseteq U$. By Proposition~\ref{proposition:positive.type}(ii), since $G$ is unitarily representable, there exist $\delta \in \R_{>0}$ and a continuous function $f \colon G \to \C$ of positive type such that $\{ g \in G \mid \vert f(g) \vert \geq (1-\delta)f(e) \} \subseteq V$. Due to~(i), \begin{displaymath}
	f' \colon \, S(\mu,G) \, \longrightarrow \, \C, \quad a \, \longmapsto \, \tfrac{1}{\mu(1)}\sum\nolimits_{g \in G} f(g)\mu(a(g))
\end{displaymath} is a continuous function of positive type with $f'(e_{S(\mu,G)}) =  f(e)$. Furthermore, if $a \in S(\mu,G)$ and $\vert f'(a) \vert \geq \left(1-\tfrac{\delta \epsilon}{\mu(1)}\right)\! f(e)$, then \begin{align*}
	\left(1-\tfrac{\delta \epsilon}{\mu(1)}\right)\! f(e) \, &\leq \, \vert f'(a) \vert \, \leq \, \tfrac{1}{\mu(1)}\sum\nolimits_{g \in G\setminus V} \vert f(g)\vert \mu(a(g)) + \tfrac{1}{\mu(1)} \sum\nolimits_{g \in V} \vert f(g) \vert \mu(a(g)) \\
	&\stackrel{\ref{lemma:positive.type}(i)}{\leq} \, \tfrac{1}{\mu(1)}\sum\nolimits_{g \in G\setminus V} (1-\delta)f(e)\mu(a(g)) + \tfrac{1}{\mu(1)}\sum\nolimits_{g \in V} f(e)\mu(a(g)) \\
	& \hspace{2.8mm} = \, \tfrac{1}{\mu(1)}(1-\delta)f(e)\mu(a[G\setminus V]) + \tfrac{1}{\mu(1)}f(e)(\mu(1)-\mu(a[G\setminus V])) \\
	& \hspace{2.8mm} = \, \tfrac{f(e)}{\mu(1)}( \mu(1) - \delta \mu(a[G\setminus V])) ,
\end{align*} therefore $\mu(1) - \delta \epsilon \leq \mu(1) - \delta \mu(a[G\setminus V])$, and so $\mu(a[G\setminus V]) \leq \epsilon$. That is, \begin{displaymath}
	\left\{ a \in S(\mu,G) \left\vert \, \vert f'(a) \vert \geq \!\left(1-\tfrac{\delta \epsilon}{\mu(1)}\right)\! f'(e_{S(\mu,G)}) \right\} \! \right. \, \subseteq \, N_{\mu}(V,\epsilon) \, \subseteq \, U .
\end{displaymath} Hence, $S(\mu,G)/\,\overline{\!\{e_{S(\mu,G)}\}\!}\,$ is unitarily representable by Proposition~\ref{proposition:positive.type}(ii). \end{proof}

Now everything is prepared to prove the corollaries mentioned earlier.

\begin{cor}\label{corollary:exotic} Let $\phi$ be a submeasure and let $G$ be a non-exotic (for example, non-trivial discrete) topological group. The following are equivalent. \begin{itemize}
	\item[(i)] $\phi$ is pathological.
	\item[(ii)] $S(\phi,G)$ is exotic.
	\item[(iii)] $L^{0}(\phi,G)$ is exotic.
	\item[(iv)] $S(\phi,G)$ is strongly exotic.
	\item[(v)] $L^{0}(\phi,G)$ is strongly exotic.
\end{itemize} \end{cor}

\begin{proof} (i)$\Longrightarrow$(iv). This is due to Theorem~\ref{theorem:exotic}.
	
(iv)$\Longrightarrow $(v). As $S(\phi,G)$ is dense in $L^{0}(\phi,G)$ due to Remark~\ref{remark:completion}, this implication follows by Remark~\ref{remark:strongly.exotic}.
	
(v)$\Longrightarrow $(iii). Trivial.
	
(iii)$\Longrightarrow$(ii). Let $H$ be a Hilbert space and let $\pi \colon S(\phi,G) \to \U(H)$ be a strongly continuous homomorphism. 
As $S(\phi,G)$ is dense in~$L^{0}(\phi,G)$ by Remark~\ref{remark:completion} and $\U(H)$ is Ra\u{\i}kov complete due to Remark~\ref{remark:isometry.groups.complete}(i), Proposition~\ref{proposition:extension.to.completion}(ii) asserts the existence of a continuous homomorphism $\pi' \colon L^{0}(\phi,G) \to \U(H)$ such that ${\pi'}\vert_{S(\phi,G)} = \pi$. Due to~(iii), the map $\pi'$ must be constant, so $\pi$ is constant. 

(ii)$\Longrightarrow$(i). Suppose that $\phi$ is not pathological. Let $\phi$ be defined on a Boolean algebra $\mathcal{A}$. Then there exists a measure $\mu$ on $\mathcal{A}$ such that $0 \ne \mu \leq \phi$. Since $S(\mu,G)$ is non-exotic by Lemma~\ref{lemma:representation}(ii) and $S(\phi,G) \to S(\mu,G), \, a \mapsto a$ is a continuous surjective homomorphism, $S(\phi,G)$ is not exotic either. \end{proof}

\begin{cor}\label{corollary:exotic.2} Let $\phi$ be a submeasure. Then the following are equivalent. \begin{itemize}
	\item[(i)] $\phi$ is pathological.
	\item[(ii)] $\mathcal{D}_{\phi}$ is exotic.
	\item[(iii)] $\mathcal{D}_{\phi}$ is strongly exotic.
\end{itemize} \end{cor}

\begin{proof} Since $\mathcal{D}_{\phi} \cong S(\phi,\Z_{2})$, this follows from Corollary~\ref{corollary:exotic}. \end{proof}

\begin{cor}\label{corollary:exotic.3} Let $G$ be a topological group, let $H$ be a unitarily representable topological group, let $\phi$ be a pathological submeasure, and let $\mu$ be a measure. Then any continuous homomorphism from $L^{0}(\phi,G)$ to $L^{0}(\mu,H)$ is trivial. \end{cor}

\begin{proof} Since $H$ is unitarily representable, $S(\mu,H)/\,\overline{\!\{e_{S(\mu,H)}\}\!}\,$ is unitarily representable due to Lemma~\ref{lemma:representation}(iii) and hence $L^{0}(\mu,H)$ is unitarily representable by Remark~\ref{remark:completion} and Remark~\ref{remark:unitarily.representable}. Now, the desired conclusion follows by Theorem~\ref{theorem:exotic}. \end{proof} 

Note that Theorem~\ref{theorem:escape} does not follow from Corollary~\ref{corollary:exotic.3}: for instance, if $H$ is an exotic Banach--Lie group (as exhibited in~\cite{banaszczyk}), then the former applies, whereas the latter does not. On the other hand, Corollary~\ref{corollary:exotic.3} is not an immediate consequence of Theorem~\ref{theorem:escape}, as substantiated by Proposition~\ref{proposition:pestov}.

\appendix

\section{Submeasures}\label{appendix:submeasures}

In this section, we compile a number of basic definitions and facts concerning submeasures on Boolean algebras, in particular pathological ones.

\begin{definition} Let $\mathcal{A}$ be a Boolean algebra. As customary, for $A,B \in \mathcal{A}$, we let \begin{displaymath}
	A \mathbin{\triangle} B \, \defeq \, (A \vee B) \wedge \neg(A \wedge B) \, = \, (A\wedge \neg B) \vee (B \wedge \neg A) .
\end{displaymath} Given $A \in \mathcal{A}$ and $\mathcal{Q} \in \Pfin (\mathcal{A})$, we will write $A = \bigveedot \mathcal{Q}$ if \begin{enumerate}
	\item[---] $A = \bigvee \mathcal{Q}$, and
	\item[---] $P \wedge Q = 0$ for any two distinct $P,Q \in \mathcal{Q}$.
\end{enumerate} A \emph{finite partition of an element $A \in \mathcal{A}$} is a finite subset $\mathcal{Q} \subseteq \mathcal{A}\setminus \{ 0 \}$ with $A = \bigveedot \mathcal{Q}$. A \emph{finite partition of unity in $\mathcal{A}$} is a finite partition of $1 \in \mathcal{A}$. The set of all finite partitions of unity in $\mathcal A$ will be denoted by $\Pi ({\mathcal A})$. For any $\mathcal{Q},\mathcal{Q}' \in \Pi(\mathcal{A})$, we define \begin{displaymath}
		\mathcal{Q} \preceq \mathcal{Q}' \quad :\Longleftrightarrow \quad \forall Q' \in \mathcal{Q}' \, \exists Q \in \mathcal{Q} \colon \ Q' \leq Q .
\end{displaymath}\end{definition}

\begin{remark}\label{remark:directed.partitions} If $\mathcal{A}$ is a Boolean algebra, then $(\Pi(\mathcal{A}),{\preceq})$ is a directed set. \end{remark}

\begin{definition} Let $\mathcal{A}$ be a Boolean algebra. A function $\phi \colon \mathcal{A} \to \mathbb{R}$ is called \begin{itemize}
		\item[---] \emph{monotone} if $\phi (A) \leq \phi (B)$ for all $A,B \in \mathcal{A}$ with $A \leq B$,
		\item[---] \emph{strictly positive} if $\phi (A) > 0$ for all $A \in \mathcal{A}\setminus \{ 0 \}$,
		\item[---] \emph{subadditive} if $\phi (A \vee B) \leq \phi (A) + \phi (B)$ for all $A,B \in \mathcal{A}$,
		\item[---] a \emph{submeasure} if $\phi$ is monotone and subadditive and satisfies $\phi(0) = 0$,
		\item[---] \emph{submodular} if $\phi (A \vee B) + \phi(A \wedge B) \leq \phi (A) + \phi (B)$ for all $A,B \in \mathcal{A}$,
		\item[---] \emph{additive} if $\phi (A \vee B) = \phi (A) + \phi (B)$ for all $A,B \in \mathcal{A}$ with $A \wedge B = 0$,
		\item[---] a \emph{measure} if $\phi$ is monotone and additive,
		\item[---] \emph{pathological} if there does not exist a non-zero measure $\mu \colon \mathcal{A} \to \mathbb{R}$ with $\mu \leq \phi$,
		\item[---] \emph{diffuse} if, for every $\epsilon \in \R_{>0}$, there is $\mathcal{Q} \in \Pi(\mathcal{A})$ with $\sup_{Q \in \mathcal{Q}} \phi(Q) \leq \epsilon$.
	\end{itemize} Suppose that $\mathcal{A}$ is \emph{$\sigma$-complete}, i.e., every countable subset of $\mathcal{A}$ admits a supremum\footnote{equivalently, infimum~\cite[Remark~314B(c)]{Fremlin3}} in the partially ordered set $\mathcal{A}$. A measure $\mu$ on $\mathcal{A}$ is said to be \emph{$\sigma$-additive} if, for any increasing sequence $(A_{n})_{n \in \N} \in \mathcal{A}^{\N}$, \begin{displaymath}
		\mu \! \left( \bigvee\nolimits_{n \in \N} A_{n} \right)\! \, = \, \sup\nolimits_{n \in \N} \mu(A_{n}) .
\end{displaymath} \end{definition}

Before proceeding to a slightly more detailed account on pathological submeasures, we recall a useful consequence of the Loomis--Sikorski representation theorem~\cite{loomis,sikorski} for $\sigma$-complete Boolean algebras, which involves the following standard constructions.

\begin{remark}\label{remark:measure.zero.ideal} Let $\mathcal{A}$ be a Boolean algebra. \begin{itemize}
		\item[(i)] Suppose that $\mathcal{N}$ is an ideal of $\mathcal{A}$, i.e., a non-empty subset $\mathcal{N} \subseteq \mathcal{A}$ such that \begin{enumerate}
			\item[---] $\{ B \in \mathcal{A} \mid B \leq A \} \subseteq \mathcal{N}$ for each $A \in \mathcal{N}$, and
			\item[---] $A \vee B \in \mathcal{N}$ for all $A,B \in \mathcal{N}$.
		\end{enumerate}	Then $\mathcal{A}/\mathcal{N} \defeq \{ A \mathbin{\triangle} \mathcal{N} \mid A \in \mathcal{A} \}$, equipped with the order inherited from~$\mathcal{A}$, constitutes a Boolean algebra (see, for instance,~\cite[Proposition~312L]{Fremlin3}). Provided that $\mathcal{A}$ is $\sigma$-complete and $\mathcal{N}$ is a $\sigma$-ideal of $\mathcal{A}$, in the sense that, moreover, $\bigvee \mathcal{C} \in \mathcal{N}$ for any countable subset $\mathcal{C} \subseteq \mathcal{N}$, the Boolean algebra~$\mathcal{A}/\mathcal{N}$ will be $\sigma$-complete, too (see, e.g.,~\cite[Proposition~314C]{Fremlin3}).
		\item[(ii)] Let $\mu$ be a measure on $\mathcal{A}$. Then $\mathcal{N}_{\mu} \defeq \{ A \in \mathcal{A} \mid \mu(A) = 0\}$ is an ideal of $\mathcal{A}$ and the well-defined map \begin{displaymath}
			\qquad \mu' \colon \, \mathcal{A}/\mathcal{N}_{\mu} \, \longrightarrow \, \R, \quad A \mathbin{\triangle} \mathcal{N}_{\mu} \, \longmapsto \, \mu(A)
		\end{displaymath} is a strictly positive measure. Moreover, if $\mathcal{A}$ is $\sigma$-complete and $\mu$ is $\sigma$-additive, then $\mathcal{N}_{\mu}$ is a $\sigma$-ideal of $\mathcal{A}$ and $\mu'$ is $\sigma$-additive as well. As customary, given any $A \in \mathcal{A}/\mathcal{N}_{\mu}$, we usually write $\mu(A)$ in place of $\mu'(A)$.
\end{itemize} \end{remark}

\begin{prop}\label{proposition:loomis.sikorski} Let $\mu$ be a measure on a Boolean algebra $\mathcal{A}$. Then there exists a finite measure space $(X,\mathcal{B},\nu)$ and a Boolean algebra homomorphism $\beta \colon \mathcal{A} \to \mathcal{B}/\mathcal{N}_{\nu}$ such that $\mu(A) = \nu(\beta(A))$ for all $A \in \mathcal{A}$. \end{prop}

\begin{proof} Due to~\cite[Corollary~392I]{Fremlin3}, there exist a $\sigma$-additive measure $\mu'$ on a $\sigma$-complete Boolean algebra $\mathcal{A}'$ and a homomorphism $\alpha \colon \mathcal{A} \to \mathcal{A}'$ such that $\mu = {\mu'} \circ \alpha$. By the Loomis--Sikorski representation theorem~\cite{loomis,sikorski} (see also~\cite[II, \S29.1, p.~117]{SikorskiBook}), there exist a measurable space~$(X,\mathcal{B})$, a $\sigma$-ideal $\mathcal{N}$ of $\mathcal{B}$, and an isomorphism $\theta \colon \mathcal{B}/\mathcal{N} \to \mathcal{A}'$. It follows that $\nu \colon \mathcal{B} \to \R, \, B \mapsto \mu'(\theta(B \mathbin{\triangle} \mathcal{N}))$ is a $\sigma$-additive measure on $\mathcal{B}$, that is, $(X,\mathcal{B},\nu)$ is a finite measure space. Since $\mathcal{N} \subseteq \mathcal{N}_{\nu}$, there is a well-defined Boolean algebra homomorphism given by \begin{displaymath}
	\pi \colon \, \mathcal{B}/\mathcal{N} \, \longrightarrow \, \mathcal{B}/\mathcal{N}_{\nu}, \quad B \mathbin{\triangle} \mathcal{N} \, \longmapsto \, B \mathbin{\triangle} \mathcal{N}_{\nu} .
\end{displaymath} Consequently, $\beta \defeq \pi \circ {\theta^{-1}} \circ \alpha \colon \mathcal{A} \to \mathcal{B}/\mathcal{N}_{\nu}$ is a homomorphism of Boolean algebras. Furthermore, for every $A \in \mathcal{A}$, \begin{displaymath}
	\nu(\beta(A)) \, = \, \nu(\pi(\theta^{-1}(\alpha(A))) \, = \, \mu'(\alpha(A)) \, = \, \mu(A) . \qedhere
\end{displaymath} \end{proof}

For the rest of this section, we focus on pathological submeasures. The following diffuseness criterion by means of Boolean algebra homomorphisms will be useful.

\begin{lem}\label{lemma:diffuse} A submeasure $\phi$ on a Boolean algebra $\mathcal{A}$ is diffuse if and only if, for every homomorphism $\chi \colon \mathcal{A} \to 2$ and every $r \in \R_{>0}$, one has $r\chi \nleq \phi$. \end{lem}

\begin{proof} Let $\phi$ be a submeasure on a Boolean algebra $\mathcal{A}$.
	
($\Longrightarrow$) Assume that $\phi$ is diffuse. Consider any $r \in \R_{>0}$. Since $\phi$ is diffuse, there exists $\mathcal{Q} \in \Pi(\mathcal{A})$ such that $\sup \{ \phi(Q) \mid Q \in \mathcal{Q}\} \leq \tfrac{r}{2}$. Now, if $\chi \colon \mathcal{A} \to 2$ is a homomorphism, then there exists some $Q \in \mathcal{Q}$ with $\chi(Q) = 1$. Then $(r\chi)(Q) = r > \tfrac{r}{2} \geq \phi(Q)$, which entails that $r\chi \nleq \phi$, as desired.
	
($\Longleftarrow$) Suppose that $\phi$ is not diffuse. Then there exists $r \in \R_{>0}$ such that, for every $\mathcal{Q} \in \Pi(\mathcal{A})$, there is $B \in \mathcal{Q}$ with $\phi (B) \geq r$. Given $\mathcal{Q},\mathcal{Q}' \in \Pi(\mathcal{A})$ with $\mathcal{Q}' \preceq \mathcal{Q}$, let $\pi_{\mathcal{Q},\mathcal{Q}'} \colon \mathcal{Q} \to \mathcal{Q}'$ denote the unique map such that $Q \leq \pi_{\mathcal{Q},\mathcal{Q}'}(Q)$ for each $Q \in \mathcal{Q}$. Consider the topological product space $X \defeq \prod_{\mathcal{Q} \in \Pi(\mathcal{A})} \mathcal{Q}$, where each of the finite sets $\mathcal{Q} \in \Pi(\mathcal{A})$ carries the discrete topology. For every $\mathcal{Q} \in \Pi(\mathcal{A})$, \begin{displaymath}
	Y(\mathcal{Q}) \, \defeq \, \{ x \in X \mid \phi(x(\mathcal{Q})) \geq r , \, \forall \mathcal{Q}' \in \Pi(\mathcal{A})\colon \, \mathcal{Q}' \preceq \mathcal{Q} \, \Longrightarrow \, \pi_{\mathcal{Q},\mathcal{Q}'}(x(\mathcal{Q})) = x(\mathcal{Q}') \} 
\end{displaymath} is a non-empty closed subset of $X$. Note that $Y(\mathcal{Q}) \subseteq Y(\mathcal{Q}')$ for any two $\mathcal{Q},\mathcal{Q}' \in \Pi(\mathcal{A})$ with $\mathcal{Q}' \preceq \mathcal{Q}$. Hence, Remark~\ref{remark:directed.partitions} implies that $\mathcal{Y} \defeq \{ Y(\mathcal{Q}) \mid \mathcal{Q} \in \Pi(\mathcal{A}) \}$ has the finite-intersection property. Since $X$ is compact, thus $\bigcap \mathcal{Y} \ne \emptyset$. Pick any element $x \in \bigcap \mathcal{Y}$ and define \begin{displaymath}
	\chi \colon \, \mathcal{A} \, \longrightarrow \, 2, \quad A \, \longmapsto \, \begin{cases}
			\, 1 & \text{if } x(\{ A, \, \neg A \}) = A, \\
			\, 0 & \text{otherwise}.
		\end{cases}
\end{displaymath} Evidently, $\chi(\neg A) = \neg \chi(A)$ for each $A \in \mathcal{A}$. Also, $\phi(0) = 0 < r \leq \phi(x(\{ 0, \, 1 \}))$ and hence $\chi(0) = 0$. Furthermore, if $A,B \in \mathcal{A}$, then a straightforward case analysis considering \begin{displaymath}
	\mathcal{Q} \, \defeq \, \{ A \wedge B, \, A \wedge \neg B, \, B \wedge \neg A, \, \neg (A \vee B )\}\setminus \{ 0 \} \, \in \, \Pi(\mathcal{A})
\end{displaymath} and using that \begin{displaymath}
	x(\{ A,\, \neg A \}) \, = \, \pi_{\mathcal{Q}, \{ A,\, \neg A \}}(x(\mathcal{Q})), \qquad  x(\{ B,\, \neg B \}) \, = \, \pi_{\mathcal{Q}, \{ B,\, \neg B \}}(x(\mathcal{Q}))
\end{displaymath} shows that $\chi(A \wedge B) = \chi(A) \wedge \chi(B)$. Therefore, $\chi$ is a Boolean algebra homomorphism. Finally, we observe that $r\chi \leq \phi$: indeed, if $A \in \mathcal{A}$, then either $\chi(A) = 0$ and thus trivially $(r\chi)(A) \leq \phi(A)$, or else $\chi(A) = 1$ and hence $\phi(A) = \phi(x(\{ A, \, \neg A \})) \geq r$, so that $(r\chi)(A) = r \leq \phi(A)$. This completes the argument. \end{proof}

\begin{cor}\label{corollary:diffuse} Every pathological submeasure is diffuse. \end{cor}

\begin{proof} This is a direct consequence of Lemma~\ref{lemma:diffuse}. \end{proof}

The subsequent Corollary~\ref{corollary:pathological.countable}, which is a consequence of the following proposition, is used in the proof of Lemma~\ref{lemma:countable.integral}.

\begin{prop}[\cite{SchneiderSolecki}, Theorem~4.7(i)\footnote{See also \cite[Corollary~4.9]{SchneiderSolecki}, which is based on~\cite[Theorem~5]{christensen}.}]\label{proposition:christensen} A submeasure $\phi$ on a Boolean algebra $\mathcal{A}$ is pathological if and only if, for every $\epsilon \in \R_{>0}$, there exist $m \in \N_{>0}$, $\mathcal{C} = (C_{i})_{i < m} \in \mathcal{A}^{m}$ and $\mathcal{Q} \in \Pi(\mathcal{A})$ such that $\sup\nolimits_{i < m} \phi(C_{i}) \leq \epsilon$ and \begin{displaymath}
		\forall Q \in \mathcal{Q} \colon \qquad \lvert \{ i < m \mid Q \leq C_{i} \} \rvert \, \geq \, (1-\epsilon)m .
\end{displaymath} \end{prop}

\begin{cor}\label{corollary:pathological.countable} Let $\phi$ be a pathological submeasure on a Boolean algebra $\mathcal{A}$. For every countable subset $\mathcal{C} \subseteq \mathcal{A}$ there exists a countable Boolean subalgebra $\mathcal{B} \leq \mathcal{A}$ such that $\mathcal{C} \subseteq \mathcal{B}$ and $\phi\vert_{\mathcal{B}}$ is pathological. \end{cor}

\begin{proof} First of all, Proposition~\ref{proposition:christensen} entails the existence of a countable Boolean subalgebra $\mathcal{B}_{0} \leq \mathcal{A}$ such that $\phi\vert_{\mathcal{B}_{0}}$ is pathological. Now, if $\mathcal{C}$ is any countable subset of $\mathcal{A}$, then the Boolean subalgebra $\mathcal{B} \leq \mathcal{A}$ generated by $\mathcal{B}_{0} \cup \mathcal{C}$ will be countable, contain $\mathcal{C}$, and be such that $\phi\vert_{\mathcal{B}}$ is pathological. \end{proof}

\section{Ra\u{\i}kov completion}\label{appendix:completion}

In this section, we recollect some relevant background material on Ra\u{\i}kov complete topological groups. The reader is referred to~\cite[3.6]{AT} and~\cite[C8]{stroppel} for further details.

Let $G$ be a topological group. Concerning the normal subgroup $\,\overline{\!\{e \}\!} \, = \bigcap \Neigh (G) \unlhd G$, we recall that $G$ is Hausdorff if and only if $\,\overline{\!\{e \}\!} \, = \{ e \}$. Furthermore, let us  recall the following three uniformities naturally associated with the topological group $G$: the \emph{left uniformity} of $G$ is defined as \begin{displaymath}
	\{ E \subseteq G \times G \mid \exists U \in \Neigh(G) \, \forall x,y \in G \colon \, x \in yU \Longrightarrow \, (x,y) \in E \} ,
\end{displaymath} the \emph{right uniformity} of $G$ is defined as \begin{displaymath}
	\{ E \subseteq G \times G \mid \exists U \in \Neigh(G) \, \forall x,y \in G \colon \, x \in Uy \Longrightarrow \, (x,y) \in E \} ,
\end{displaymath} and the \emph{bilaterial} or \emph{two-sided uniformity} of $G$ is defined to be \begin{displaymath}
	\{ E \subseteq G \times G \mid \exists U \in \Neigh(G) \, \forall x,y \in G \colon \, x \in yU \cap Uy \Longrightarrow \, (x,y) \in E \} .
\end{displaymath} We say that $G$ is \emph{Ra\u{\i}kov complete} if $G$ is complete\footnote{See~\cite[II, \S3.3]{bourbaki} for the notion of completeness of a uniform space.} with respect to the bilateral uniformity. A fundamental result of Ra\u{\i}kov~\cite{raikov} asserts that a Hausdorff topological group is Ra\u{\i}kov complete if and only if it is \emph{absolutely closed}, in the sense that every embedding of it into another Hausdorff topological group has closed range.

\begin{prop}\label{proposition:extension.to.completion} Let $G$ and $H$ be topological groups, let $G_{0}$ be a dense topological subgroup of $G$, and let $H_{0}$ be a dense topological subgroup of $H$. Suppose that $H$ is Ra\u{\i}kov complete and Hausdorff. The following hold. \begin{itemize}
	\item[(i)] Suppose that $f \colon G_{0} \to H$ is bilaterally uniformly continuous. Then there exists a unique continuous map $f' \colon G \to H$ such that $f'\vert_{G_{0}} = f$. Moreover, $f'$ is bilaterally uniformly continuous.
	\item[(ii)] Every continuous homomorphism from $G_{0}$ to $H$ extends to a continuous homomorphism from $G$ to $H$.
	\item[(iii)] If $G$ is Ra\u{\i}kov complete and Hausdorff, then every topological isomorphism from $G_{0}$ to $H_{0}$ extends to such from $G$ to $H$.
\end{itemize} \end{prop}

\begin{proof} Point (i) is a special case of~\cite[Lemma~8.38]{stroppel}, the assertion of~(ii) is proved in~\cite[Proposition~3.6.12]{AT}, and (iii) is proved in~\cite[Proposition~3.6.13]{AT}. \end{proof}

\begin{remark}\label{remark:isometry.groups.complete} (i) Let $X$ be a metric space and consider the topological group $\Iso(X)$ consisting of all (surjective) isometries of $X$, equipped with the topology of pointwise convergence. If $X$ is complete, then $\Iso(X)$ and therefore any of its closed subgroups will be complete with respect to the left uniformity, hence Ra\u{\i}kov complete. In particular, if $H$ is a Hilbert space, then its unitary group $\U(H)$, endowed with the strong operator topology, will be Ra\u{\i}kov complete.

(ii) According to~\cite[Theorem~4.3.26]{engelking} and~\cite[Theorem~4.3.7]{AT}, every completely metrizable topological group is Ra\u{\i}kov complete. \end{remark}

Finally in this section, we turn to completions. To this end, let $G$ be a topological group. The Hausdorff completion\footnote{See~\cite[II, \S3.7]{bourbaki} or~\cite{robertson} for the Hausdorff completion of a uniform space.} $\widehat{G}$ of $G$ with respect to its bilateral uniformity naturally turns into a topological group, which is called the \emph{Ra\u{\i}kov completion} of~$G$. We briefly sketch this construction, essentially following~\cite[C8]{stroppel}. Any bilaterally Cauchy filter $\mathcal{F}$ on $G$ contains a unique minimal bilaterally Cauchy filter~\cite[Lemma~8.30(f)]{stroppel}, which we denote by $\mathcal{F}_{\bot}$. The topological group $\widehat{G}$ consists of the set of all minimal bilaterally Cauchy filters on $G$, equipped with the multiplication defined by \begin{displaymath}
	\mathcal{F} \cdot \mathcal{F}' \, \defeq \, \{ E \subseteq G \mid \exists (F,F') \in \mathcal{F} \times \mathcal{F}' \colon \, FF' \subseteq E \}_{\bot} \qquad \left(\mathcal{F},\mathcal{F}' \in \widehat{G}\right)
\end{displaymath} and endowed the group topology uniquely determined by the neighborhood basis \begin{displaymath}
	\left. \left. \left\{ \left\{ \mathcal{F} \in \widehat{G} \, \right\vert U \in \mathcal{F} \right\} \, \right\vert U \in \Neigh(G) \right\} 
\end{displaymath} at the neutral element $e_{\widehat{G}} = \Neigh(G)$ (see~\cite[Theorem~8.44, Remarks~8.45]{stroppel}). The topological group $\widehat{G}$ is Hausdorff~\cite[Lemma~8.34]{stroppel} and Ra\u{\i}kov complete~\cite[Lemma~8.35]{stroppel}. We summarize some properties of the accompanying completion map as follows.

\begin{remark}\label{remark:completion} Let $G$ be a topological group. The natural map \begin{displaymath}
	\iota_{G} \colon \, G \, \longrightarrow \, \widehat{G}, \quad g \, \longmapsto \, \Neigh_{g}(G)
\end{displaymath} is a continuous homomorphism with $\overline{\im (\iota_{G})} = \widehat{G}$ (see~\cite[Lemma~8.36]{stroppel}) and \begin{displaymath}
	\ker (\iota_{G}) \, = \, \{ g \in G \mid \Neigh_{g}(G) = \Neigh(G) \} \, = \, \bigcap \Neigh (G) \, = \, \, \overline{\!\{ e \}\!}\, ,
\end{displaymath} and the induced homomorphism \begin{displaymath}
	G/\,\overline{\!\{e \}\!} \, \longrightarrow \, \widehat{G}, \quad g \, \overline{\!\{e \}\!}\, \, \longmapsto \, \iota_{G}(g)
\end{displaymath} is a topological embedding (cf.~\cite[Lemma~8.33]{stroppel}). \end{remark}

\section{An alternative argument}

For readers interested primarily in the results of Section~\ref{section:exoticness}, but wishing to avoid the use of escape functions, we record a simplified version of Theorem~\ref{theorem:escape}, which suffices for the application in the proof of Lemma~\ref{lemma:exotic}.

For the sake of completeness, let us recall that a topological group $G$ is said to have \emph{no small subgroups} if $\trap(U) = \{ e \}$ for some $U \in \Neigh(G)$. The statement and proof of the following lemma make use of the identification suggested by Remark~\ref{remark:moore}.

\begin{lem}\label{lemma:alternative} Let $H$ be a Polish group with no small subgroups, let $G$ be a countable amenable group, let $(X,\mathcal{B},\mu)$ be a finite measure space, let $\phi$ be a submeasure on a countable Boolean algebra $\mathcal{A}$, and let $\pi \colon S(\phi,G) \to L^{0}(\mu,H)$ be a continuous homomorphism. Then \begin{displaymath}
	\theta \colon \, \mathcal{A} \, \longrightarrow \, \mathcal{B}/\mathcal{N}_{\mu}, \quad A \, \longmapsto \, \bigvee\nolimits_{f \in \pi(\Gamma(A))} f^{-1}(H\setminus \{ e \})
\end{displaymath} is a $\phi$-to-$\mu$-continuous $\vee$-monoid homomorphism. If $\phi$ is pathological, then $\pi$ is trivial. \end{lem}

\begin{proof} Clearly, $\theta(0) = 0$ by Lemma~\ref{lemma:supported.subgroups}(i). For all $A,B \in \mathcal{A}$, \begin{align*}
	\theta(A \vee B) \, &= \, \bigvee\nolimits_{f \in \pi(\Gamma(A \vee B))} f^{-1}(H\setminus \{ e \}) \\
		& \stackrel{\ref{lemma:supported.subgroups}(ii)}{=} \, \bigvee\nolimits_{f \in \pi(\Gamma(A)\Gamma(B))} f^{-1}(H\setminus \{e\}) \\
		& \stackrel{\pi\,\text{hom.}}{=} \bigvee\nolimits_{f \in \pi(\Gamma(A))\pi(\Gamma(B))} f^{-1}(H\setminus \{e\}) \\
		& = \, \bigcup\nolimits_{f \in \pi(\Gamma(A))} \bigcup\nolimits_{g \in \pi(\Gamma(B))} (f\cdot g)^{-1}(H\setminus \{e\}) \\
		& = \, \bigvee\nolimits_{f \in \pi(\Gamma(A))} \bigcup\nolimits_{g \in \pi(\Gamma(B))} f^{-1}(H\setminus \{e\}) \vee g^{-1}(H\setminus \{e\}) \\
		& = \, \left( \bigvee\nolimits_{f \in \pi(\Gamma(A))} f^{-1}(H\setminus \{e\}) \right) \vee \left( \bigvee\nolimits_{g \in \pi(\Gamma(B))} g^{-1}(H\setminus \{e\}) \right) \\
		& = \, \theta(A) \vee \theta (B) .
\end{align*} This means that $\theta$ is indeed a $\vee$-monoid homomorphism.
	
We proceed to showing that $\theta$ is $\phi$-to-$\mu$-continuous, which amounts to proving that, for each $\epsilon \in \R_{>0}$, there exists $\delta \in \R_{>0}$ such that \begin{equation}\label{E:dep.alternative}
	\forall A \in \mathcal{A} \colon \qquad \phi(A) \leq \delta \ \Longrightarrow \ \mu(\theta(A)) \leq \epsilon . 
\end{equation} To this end, fix $\epsilon \in \R_{>0}$. Since $H$ has no small subgroups, we find an open identity neighborhood $U_{0}$ in $H$ with $\trap(U_{0}) = \{ e \}$. Let $U_{1}$ be an open identity neighborhood in $H$ such that $U_{1}U_{1}^{-1} \subseteq U_{0}$. As $\pi$ is continuous, there exists $\delta \in \R_{>0}$ such that \begin{equation}\label{representation.continuity}
	\forall A \in \mathcal{A} \colon \quad \phi(A) \leq \delta \ \Longrightarrow \ \!\left( \forall f \in \pi(\Gamma(A)) \colon \ \mu\!\left(f^{-1}(H\setminus U_{1})\right)\! \leq \tfrac{\epsilon}{4} \right) .
\end{equation} We show that~\eqref{E:dep.alternative} holds for this choice of $\delta$. To this end, fix $A \in \mathcal{A}$ with $\phi(A) \leq \delta$.

Note that $\Gamma(A)$ is countable and amenable as a discrete group by Remark~\ref{remark:inductive.limit}. In view of Remark~\ref{remark:moore}, for each $i \in \{ 0,1 \}$ we obtain a well-defined map \begin{displaymath}
	\zeta_{i} \colon \, L^{0}(\mu,H) \, \longrightarrow \, L^{0}(\mu,\R), \quad f \, \longmapsto \, \chi_{f^{-1}(H\setminus U_{i})},
\end{displaymath} where we let \begin{displaymath}
	\chi_{B} \colon \, X \, \longrightarrow \, \{ 0,1 \}, \quad x \, \longmapsto \, \begin{cases}
		\, 1 & \text{if } x \in B, \\
		\, 0 & \text{otherwise}
	\end{cases}
\end{displaymath} for any subset $B \subseteq X$. Since $\pi$ is a homomorphism and $U_{1}U_{1}^{-1} \subseteq U_{0}$, we see that the pair of maps $\rho_{i} \colon \Gamma(A) \to L^{0}(\mu,\R), \, a \mapsto \zeta_{i}(\pi(a))$ $(i \in \{ 0,1 \})$ satisfies the hypothesis of Lemma~\ref{lemma:abstract.key}: indeed, if $a,b \in \Gamma(A)$, then, for $\mu$-almost every $x \in X$, either $\pi\!\left(ab^{-1}\right)\!(x) \in U_{0}$ and hence \begin{displaymath}
	\rho_{0}\!\left(ab^{-1}\right)\!(x) \, = \, 0 \, \leq \, \rho_{1}(a)(x) + \rho_{1}(b)(x) ,
\end{displaymath} or $\pi(a)(x)(\pi(b)(x))^{-1} = \pi\!\left(ab^{-1}\right)\!(x) \notin U_{0}$ and thus $\{ \pi(a)(x), \pi(b)(x) \} \nsubseteq U_{1}$, so \begin{displaymath}
	\rho_{0}\!\left(ab^{-1}\right)\!(x) \, = \, 1 \, \leq \, \rho_{1}(a)(x) + \rho_{1}(b)(x) .
\end{displaymath} Consequently, Lemma~\ref{lemma:abstract.key} asserts that \begin{equation}\label{abstract.key.estimate.alternative}
	\mu\!\left(\bigvee\nolimits_{f \in \pi(\Gamma(A))} f^{-1}(H\setminus U_{0})\right) \!\, \leq \, 4\sup\nolimits_{a\in \pi(\Gamma(A))}\mu\!\left(f^{-1}(H\setminus U_{1})\right).
\end{equation} Since \begin{align*}
	\theta (A) \, &= \, \bigvee\nolimits_{f \in \pi(\Gamma(A))} f^{-1}(H\setminus \{ e \}) \\
	& \stackrel{\trap(U_{0}) = \{ e \}}{\leq} \, \bigvee\nolimits_{f \in \pi(\Gamma(A))} \bigvee\nolimits_{g \in \langle f \rangle} g^{-1}(H\setminus U_{0}) \, = \, \bigvee\nolimits_{f \in \pi(\Gamma(A))} f^{-1}(H\setminus U_{0}) ,
\end{align*} we conclude that \begin{align*}
	\mu(\theta(A)) \, &\leq \, \mu\!\left(\bigvee\nolimits_{f \in \pi(\Gamma(A))} f^{-1}(H\setminus U_{0}) \right)\\
	&\stackrel{\eqref{abstract.key.estimate.alternative}}{\leq} \, 4\sup\nolimits_{a\in \pi(\Gamma(A))}\mu\!\left(f^{-1}(H\setminus U_{1})\right) \! \, \stackrel{\eqref{representation.continuity}}{\leq} \, \epsilon ,
\end{align*} as desired in~\eqref{E:dep.alternative}, thus proving $\phi$-to-$\mu$-continuity of $\theta$.

Finally, suppose that $\phi$ is pathological. Since $\theta \colon \mathcal{A} \to \mathcal{B}/\mathcal{N}_{\mu}$ is a $\phi$-to-$\mu$-continuous $\vee$-monoid homomorphism, Lemma~\ref{lemma:christensen.kelley} asserts that \begin{displaymath}
	\mu\!\left( \bigcup\nolimits_{f \in \im(\pi)} f^{-1}(H\setminus \{e\}) \right)\! \, = \, \mu(\theta (1)) \, = \, 0 ,
\end{displaymath} thus $\pi(a) = e_{L^{0}(\mu,H)}$ for every $a \in S(\phi,G)$, i.e., $\pi$ is trivial. \end{proof}

\begin{thm}\label{theorem:alternative} Let $G$ be a topological group, let $\phi$ be a pathological submeasure, and let $\mu$ be a measure. If $H$ is a Polish group with no small subgroups, then any continuous homomorphism from $L^{0}(\phi,G)$ to $L^{0}(\mu,H)$ is trivial. In particular, any continuous homomorphism from $L^{0}(\phi,G)$ to $L^{0}(\mu,\T)$ is trivial. \end{thm}

\begin{proof} The argument is similar to the one in the proof of Theorem~\ref{theorem:escape}: upon invoking Remark~\ref{remark:completion}, Lemma~\ref{lemma:countable.integral}, Proposition~\ref{proposition:loomis.sikorski}, Remark~\ref{remark:boolean.transformations}, the first assertion is a consequence of Lemma~\ref{lemma:alternative}. (We omit the details to avoid repetition.) Since $\T$ is a Polish group with no small subgroups, the second assertion follows at once. \end{proof}

\section*{Acknowledgments}

The first-named author is indebted to Maxime Gheysens for a helpful discussion on unitarizability of amenable topological groups as well as for suggesting Terence Tao's book~\cite{TaoBook} as a reference to relevant aspects of the structure theory of locally compact groups. Furthermore, the authors would like to express their sincere gratitude towards Vladimir Pestov for several insightful comments and helpful references, as well as for his permission to include Proposition~\ref{proposition:pestov} in the present manuscript. Finally, the authors wish to thank the anonymous referee for a number of suggestions that helped to improve the presentation of this work.


\begin{thebibliography}{50}
	
	\def\cprime{$'$}
	
	\bibitem{AT}
	Alexander Arhangel\cprime skii and Mikhail Tkachenko, \emph{Topological groups and related structures}. Atlantis Studies in Mathematics, 1. Atlantis Press, Paris, 2008.
	
	\bibitem{atkin}
	Christopher~J. Atkin, \emph{Boundedness in uniform spaces, topological groups, and homogeneous spaces}. Acta Math. Hungar.~\textbf{57} (1991), no.~3--4, pp.~213--232.
	
	\bibitem{banaszczyk}
	Wojciech Banaszczyk, \emph{On the existence of exotic Banach-Lie groups}. Math. Ann.~\textbf{264} (1983), no.~4, pp.~485--493.
	
	\bibitem{banaszczyk2}
	Wojciech Banaszczyk, \emph{On the existence of commutative Banach-Lie groups which do not admit continuous unitary representations}. Colloq. Math.~\textbf{52} (1987), no.~1, pp.~113--118.
		
	\bibitem{BekkaEtAlBook}
	Bachir Bekka, Pierre de la Harpe, Alain Valette, \emph{Kazhdan's property (T)}. New Math. Monogr., 11, Cambridge University Press, Cambridge, 2008.	
		
	\bibitem{birkhoff}
	Garret Birkhoff, \emph{A note on topological groups}. Compositio Math.~\textbf{3} (1936), pp.~427--430.	
		
	\bibitem{BlecherGoldsteinLabuschagne}
	David~P. Blecher, Stanisław Goldstein, Louis~E. Labuschagne, \emph{Abelian von Neumann algebras, measure algebras and $L^{\infty}$-spaces}. Expo. Math.~\textbf{40} (2022), no.~3, pp.~758--818.	
		
	\bibitem{bourbaki}
	Nicolas Bourbaki, \emph{General topology. Chapters 1--4}. Translated from the French. Reprint of the 1989 English translation, Elem. Math., Springer-Verlag, Berlin, 1998.
	
	\bibitem{CarderiThom}
	Alessandro Carderi, Andreas Thom, \emph{An exotic group as limit of finite special linear groups}. Ann. Inst. Fourier (Grenoble)~\textbf{68} (2018), no.~1, pp.~257--273.	
		
	\bibitem{christensen}
	Jens~P.~R. Christensen, \emph{Some results with relation to the control measure problem}. In \emph{Vector Space Measures and Applications II}, Lecture Notes in Math.~\textbf{645}, Springer, 1978, pp.~27--34. 
	
	\bibitem{DaviesRogers}
	Roy O. Davies, C.~Ambrose Rogers, \emph{The problem of subsets of finite positive measure}. Bull. London Math. Soc.~\textbf{1} (1969), pp.~47--54.
	
	\bibitem{day}
	Mahlon~M. Day, \emph{Means for the bounded functions and ergodicity of the bounded representations of semi-groups}. Trans. Amer. Math. Soc.~\textbf{69} (1950), pp.~276--291.
	
	\bibitem{dixmier}
	Jacques Dixmier, \emph{Les moyennes invariantes dans les semi-groupes et leurs applications}. Acta Sci. Math. (Szeged)~\textbf{12} (1950), pp.~213--227.
	
	\bibitem{enflo}
	Per~H. Enflo, \emph{Uniform structures and square roots in topological groups}. Israel J. Math.~\textbf{8} (1970), pp.~230--252.
	
	\bibitem{engelking}
	Ryszard Engelking, \emph{General topology}. Sigma Ser. Pure Math., 6, Heldermann Verlag, Berlin, 1989.
	
	\bibitem{ErdosHajnal}
	Paul Erd{\"o}s, Andr{\'a}s Hajnal, \emph{On chromatic graphs}. Mat. Lapok~\textbf{18} (1967), pp.~1--4. 
	
	\bibitem{FarahSolecki}
	Ilijas Farah, S{\l}awomir Solecki, \emph{Extreme amenability of $L_0$, a Ramsey theorem, and L{\'e}vy groups}. J. Funct. Anal.~\textbf{255} (2008), pp.~471--493.
	
	\bibitem{folland}
	Gerald~B. Folland, \emph{Real analysis}. Modern techniques and their applications. Second edition. Pure Appl. Math. (N. Y.), Wiley-Intersci. Publ., John Wiley \& Sons, Inc., New York, 1999.
	
	\bibitem{Fremlin2}
	David~H. Fremlin, \emph{Measure theory. Vol. 2}. Broad foundations. Corrected second printing of the 2001 original. Torres Fremlin, Colchester, 2003.
	
	\bibitem{Fremlin3}
	David~H. Fremlin, \emph{Measure theory. Vol. 3}. Measure algebras. Corrected second printing of the 2002 original. Torres Fremlin, Colchester, 2004. 
	
	\bibitem{Fremlin4}
	David~H. Fremlin, \emph{Measure theory. Vol. 4}. Topological measure spaces. Part I, II, Corrected second printing of the 2003 original. Torres Fremlin, Colchester, 2006.
	
	\bibitem{galindo}
	Jorge Galindo, \emph{On unitary representability of topological groups}. Math. Z.~\textbf{263} (2009), no.~1, pp.~211--220.
	
	\bibitem{gao}
	Su Gao, \emph{Unitary group actions and Hilbertian Polish metric spaces}. Logic and its applications, pp.~53--72. Contemp. Math., 380, American Mathematical Society, Providence, RI, 2005.
	
	\bibitem{GaoKechris}
	Su Gao, Alexander S. Kechris, \emph{On the classification of Polish metric spaces up to isometry}. Mem. Amer. Math. Soc.~\textbf{161} (2003), no.~766.
	
	\bibitem{gheysens}
	Maxime Gheysens, \emph{Representing groups against all odds}. PhD thesis, École polytechnique fédérale de Lausanne, 2017. \href{https://doi.org/10.5075/epfl-thesis-7823}{https://doi.org/10.5075/epfl-thesis-7823}
	
	\bibitem{GlasnerTsirelsonWeiss}
	Eli Glasner, Boris Tsirelson, Benjamin Weiss, \emph{The automorphism group of the Gaussian measure cannot act pointwise}. Israel J. Math.~\textbf{148} (2005). Probability in mathematics, pp.~305--329.
	
	\bibitem{GlasnerWeiss}
	Eli Glasner, Benjamin Weiss, \emph{Spatial and non-spatial actions of Polish groups}. Ergodic Theory Dyn. Syst.~\textbf{25} (2005), no.~5, pp.~1521--1538.
	
	\bibitem{gleason}
	Andrew~M. Gleason, \emph{The structure of locally compact groups}. Duke Math. J.~\textbf{18} (1951), pp.~85--104.
	
	\bibitem{hejcman}
	Jan Hejcman, \emph{Boundedness in uniform spaces and topological groups}. Czechoslovak Math. J.~\textbf{9} (84) (1959), pp.~544--563.
	
	\bibitem{HererChristensen}
	Wojciech Herer, Jens~P.~R. Christensen, \emph{On the existence of pathological submeasures and the construction of exotic topological groups}. Math. Ann.~\textbf{213} (1975), pp.~203--210.
	
	\bibitem{HilgertNeeb}
	Joachim Hilgert, Karl-Hermann Neeb, \emph{Lie semigroups and their applications}. Lecture Notes in Math., 1552, Springer-Verlag, Berlin, 1993.
	
	\bibitem{james}
	Ioan~M. James, \emph{Topological and uniform spaces}. Undergraduate Texts in Mathematics, Springer, New York, 1987.
	
	\bibitem{kelley59}
	John L. Kelley, \emph{Measures on Boolean algebras}. Pacific J. Math.~\textbf{9} (1959), pp.~1165--1177.
	
	\bibitem{kakutani}
	Shizuo Kakutani, \emph{\"Uber die Metrisation der topologischen Gruppen}. Proc. Imp. Acad.~\textbf{12} (1936), no.~4, pp.~82--84.
	
	\bibitem{KwiatkowskaSolecki}
	Aleksandra Kwiatkowska, S{\l}awomir Solecki, \emph{Spatial models of Boolean actions and groups of isometries}. Ergodic Theory Dynam. Systems~\textbf{31} (2011), no.~2, pp.~405--421.
	
	\bibitem{loomis}
	Lynn~H. Loomis, \emph{On the representation of $\sigma$-complete Boolean algebras}. Bull. Amer. Math Soc.~\textbf{53} (1947), pp.~757--760.
	
	
	\bibitem{mazur}
	Krzysztof Mazur, \emph{$F_\sigma$-ideals and $\omega_1\omega_1^*$-gaps in the Boolean algebras $P(\omega)/I$}. Fund. Math.~\textbf{138} (1991), no.~2, pp.~103--111.
	
	\bibitem{megrel}
	Michael Megrelishvili, \emph{Every semitopological semigroup compactification of the group $H_{+}[0,1]$ is trivial}. Semigroup Forum~\textbf{63} (2001), no.~3, pp.~357--370.
	
	\bibitem{moore}
	Calvin~C. Moore, \emph{Group extensions and cohomology for locally compact groups. III}. Trans. Amer. Math. Soc.~\textbf{221} (1976), no.~1, pp.~1--33.
	
	\bibitem{MorrisPestov}
	Sidney~A. Morris, Vladimir~G. Pestov, \emph{On Lie groups in varieties of topological groups}. Colloq. Math.~\textbf{78} (1998), no.~1, pp.~39--47.
	
	
	
	\bibitem{pestov98}
	Vladimir Pestov, \emph{On free actions, minimal flows, and a problem by {E}llis}. Trans. Amer. Math. Soc.~\textbf{350} (1998), no.~10, pp.~4149--4165
	
	\bibitem{Pestov07}
	Vladimir~G. Pestov, \emph{The isometry group of the Urysohn space as a L\'evy group}. Topology Appl.~\textbf{154} (2007), no.~10, pp.~2173--2184.
	
	\bibitem{pestov10}
	Vladimir~G. Pestov, \emph{Concentration of measure and whirly actions of Polish groups}. In: \emph{Probabilistic approach to geometry}, Vol.~57. Adv. Stud. Pure Math. Math. Soc. Japan, Tokyo, 2010, pp.~383--403.
	
	\bibitem{PestovSchneider}
	Vladimir~G. Pestov, Friedrich~M. Schneider, \emph{On amenability and groups of measurable maps}. J. Funct. Anal.~\textbf{273} (2017), no.~12, pp.~3859--3874.
	
	\bibitem{raikov}
	Dmitri\u{\i}~A.	Ra\u{\i}kov, ~\emph{On the completion of topological groups}. Bull. Acad. Sci. URSS. Sér. Math. [Izvestia Akad. Nauk SSSR]~\textbf{10} (1946), pp.~513--528.
	
	\bibitem{rickert}
	Neil~W. Rickert, \emph{Amenable groups and groups with the fixed point property}. Trans. Amer. Math. Soc.~\textbf{127} (1967), pp.~221--232. 
	
	\bibitem{robertson}
	Alexander~P. Robertson, Wendy Robertson, \emph{A note on the completion of a uniform space}. J. London Math. Soc.~\textbf{33} (1958), pp.~181--185.
	
	\bibitem{rosendal}
	Christian Rosendal, \emph{Lipschitz structure and minimal metrics on topological groups}. Ark. Mat.~\textbf{56} (2018), no.~1, pp.~185--206.
	
	\bibitem{sabok}
	Marcin Sabok, \emph{Extreme amenability of abelian $L_{0}$ groups}. J. Funct. Anal.~\textbf{253} (2012), no.~10, pp.~2978--2992.
		
	\bibitem{sakai}
	Sh\^{o}ichir\^{o} Sakai, \emph{{$C\sp *$}-algebras and {$W\sp *$}-algebras}. Classics in Mathematics, Reprint of the 1971 edition, Springer-Verlag, Berlin, 1998.
	
	\bibitem{SchneiderSolecki}
	Friedrich~M. Schneider, Sławomir Solecki, \emph{Concentration of measure, classification of submeasures, and dynamics of $L_{0}$}. J. Funct. Anal.~\textbf{280} (2021), no.~5, 108890.

	\bibitem{sikorski}
	Roman Sikorski, \emph{On the representation of Boolean algebras as fields of sets}. Fund. Math.~\textbf{35} (1948), pp.~247--258.
	
	\bibitem{SikorskiBook}
	Roman Sikorski, \emph{Boolean algebras}. Ergebnisse der Mathematik und ihrer Grenzgebiete, 25. Springer-Verlag, Berlin-Göttingen-Heidelberg, 1960.
	
	\bibitem{stone}
	Marshall~H. Stone, \emph{The Theory of Representation for Boolean Algebras}. Trans. Amer. Math. Soc.~\textbf{40} (1936), no.~1, pp.~37--111.
	
	\bibitem{LecturesOnVonNeumannAlgebras}
	Şerban Valentin Strătilă, László Zsidó, \emph{Lectures on von Neumann algebras. 2nd edition}. Cambridge-IISc Series. Delhi: Cambridge University Press, 2019. 
	
	\bibitem{stroppel}
	Markus Stroppel, \emph{Locally compact groups}. EMS Textbooks in Mathematics, European Mathematical Society (EMS), Zürich, 2006.
	
	\bibitem{talagrand80} 
	Michel Talagrand, \emph{A simple example of pathological submeasure}. Math. Ann.~\textbf{252} (1979/80), no.~2, pp.~97--102.
	
	\bibitem{talagrand}
	Michel Talagrand, \emph{Maharam's problem}. Ann. of Math. (2) \textbf{168} (2008), no.~3, pp.~981--1009.
	
	\bibitem{TaoBook}
	Terence Tao, \emph{Hilbert's fifth problem and related topics}. Grad. Stud. Math., 153. American Mathematical Society, Providence, RI, 2014.
	
	\bibitem{Wagon}
	Grzegorz Tomkowicz, Stan Wagon, \emph{The Banach-Tarski paradox}. Second edition. With a foreword by Jan Mycielski. Encyclopedia of Mathematics and its Applications, 163. Cambridge University Press, New York, 2016.
	
	\bibitem{usp}
	Vladimir~V. Uspenskij, \emph{On the group of isometries of the Urysohn universal metric space}. Comment. Math. Univ. Carolin.~\textbf{31} (1990), no.~1, pp.~181--182.
	
	\bibitem{yamabe}
	Hidehiko Yamabe, \emph{A generalization of a theorem of Gleason}. Ann. of Math. (2)~\textbf{58} (1953), pp.~351--365.
	
\end{thebibliography}
\end{document}